%% file: GAGA_Grothendieck_conjecture_2.0.tex
\pdfoutput=1
\newif\ifpersonal
\RequirePackage[l2tabu,orthodox]{nag} 
\documentclass[a4paper,reqno]{amsart} 
\usepackage{amsmath,amsthm,amssymb,mathrsfs,mathtools,bm,eucal,tensor,stmaryrd} 
\usepackage{microtype,fixltx2e,lmodern} 
\usepackage[utf8]{inputenc} 
\usepackage[T1]{fontenc} 
\usepackage{enumerate,comment,braket,xspace,csquotes} 

\usepackage{tikz-cd} 
\tikzcdset{cells={font=\everymath\expandafter{\the\everymath\displaystyle}}} 

\usepackage{array}
\newcolumntype{L}[1]{>{\raggedright\let\newline\\\arraybackslash\hspace{0pt}}m{#1}}
\newcolumntype{C}[1]{>{\centering\let\newline\\\arraybackslash\hspace{0pt}}m{#1}}
\newcolumntype{R}[1]{>{\raggedleft\let\newline\\\arraybackslash\hspace{0pt}}m{#1}}

\usepackage[centering,vscale=0.7,hscale=0.8]{geometry}
\usepackage[hidelinks,linktoc=all]{hyperref}
\usepackage[capitalize]{cleveref}

\numberwithin{equation}{section}
\theoremstyle{plain}
\newtheorem{thm-intro}[equation]{Theorem}
\newtheorem{thm}[equation]{Theorem}

\newtheorem{lem}[equation]{Lemma}
\newtheorem{prop}[equation]{Proposition}
\newtheorem{conj}[equation]{Conjecture}
\newtheorem{cor}[equation]{Corollary}
\newtheorem{assumption}[equation]{Assumption}

\makeatletter
\newtheorem*{rep@theorem}{\rep@title}
\newcommand{\newreptheorem}[2]{
\newenvironment{rep#1}[1]{
 \def\rep@title{#2 \ref{##1}}
 \begin{rep@theorem}}
 {\end{rep@theorem}}}
\makeatother
\newreptheorem{theorem}{Theorem}
\newreptheorem{cor}{Corollary}

\theoremstyle{definition}
\newtheorem{defin}[equation]{Definition}

\newtheorem{notation}[equation]{Notation}
\newtheorem{rem}[equation]{Remark}
 \newtheorem{warning}[equation]{Warning}
\newtheorem{eg}[equation]{Example}
\newtheorem{construction}[equation]{Construction}

\input{newcommands_all}

\begin{document}
	
\title{Descent problems for derived Azumaya algebras}

\author{Federico Binda}

\author{Mauro Porta}

\date{\today}


\maketitle

\personal{PERSONAL COMMENTS ARE SHOWN!!!}

\begin{abstract}
	This paper is dedicated to a further study of derived Azumaya algebras.
	The first result we obtain is a Beauville-Laszlo-style property for such objects (considered up to Morita equivalence), which is consequence of a more general Beauville-Laszlo kind of statement for quasi-coherent sheaves of categories.
	Next, we prove that given any (derived) scheme $X$, proper over the spectrum of a quasi-excellent Henselian ring, the derived Brauer group of $X$ injects into the one of the Henselization of $X$ along the base, generalizing a classical result of Grothendieck and a more recent theorem of Geisser-Morin.
	As a separate application, we deduce that Grothendieck's existence theorem holds for the stable $\infty$-categories of twisted sheaves even when the corresponding $\bbG_m$-gerbe does not satisfy the resolution property, offering an improvement of a result of Alper, Rydh and Hall.
\end{abstract}

\tableofcontents

\section{Introduction}
Let $X$ be a scheme.
The  Brauer-Grothendieck group of $X$, i.e. the second étale cohomology group  $\rH^2\et(X, \bbG_m)$ of the multiplicative group $\bbG_m$, is an important arithmetico-geometric invariant of the scheme: for example,  it can be used to measure obstructions to various Hasse principles for zero-cycles and rational points, thanks to the Brauer-Manin pairing; or it serves, in the form of the  unramified Brauer group, as a useful birational invariant, as shown in the construction, due to Artin and Mumford, of first example of a variety that is unirational but not rational (a counterexample to Luroth's problem in dimension 3, see \cite{Beauville_Luroth} for a survey and \cite{Colliot_Thelene_Brauer_Grothendieck_group}).  

\medskip

As it is often the case for \'etale cohomology groups, it is somewhat difficult to get a concrete grasp on the geometric meaning of classes in $\rH^2\et(X, \bbG_m)$.
Nevertheless, in the special case of the multiplicative group it is suggestive to consider the following table:

\vspace{10pt}

\begin{center}
\begin{tabular}{C{3cm}|C{3cm}|C{3cm}|C{3cm}}
	$n$ & 0 & 1 & 2 \\
	\hline
	Geometric meaning of $\rH^n\et(X,\bbG_m)$ & Invertible algebraic functions on $X$ & Invertible coherent sheaves on $X$ & ?
\end{tabular}
\end{center}

\medskip

A first attempt at filling this table has been done by Grothendieck: building on earlier works of Auslander and Goldman, he introduced In \cite[Exposé VI]{Grothendieck_Dix_expose} the notion of sheaf of Azumaya algebras on $X$: a vector bundle equipped with an associative multiplication satisfying a certain \emph{invertibility} property, generalizing the notion of central simple algebra over a field.
When considered up to Morita equivalence, Azumaya algebras form a group $\mathrm{Br}_{\mathrm{Az}}(X)$, known as the \emph{Brauer-Azumaya group} of $X$ (following e.g.\ \cite{Colliot_Thelene_Brauer_Grothendieck_group}).
Grothendieck provided a natural injective map
\[ \mathrm{Br}_{\mathrm{Az}}(X) \longrightarrow \rH^2\et(X,\bbG_m) \]
which, however, is far from being surjective in general.
Indeed, the classical Skolem-Noether theorem implies that it can only hit torsion classes, whereas D.\ Mumford constructed an example of a singular surface $S$ whose $\rH^2\et(S,\bbG_m)$ contains a non-torsion classes (on the other hand O.\ Gabber showed that the above map is surjective on torsion classes, at least when $X$ is quasi-projective).\\

Much later, the advent of homotopical techniques allowed to solve this problem in a completely satisfactory way.
In his groundbreaking work \cite{Toen_Azumaya}, B.\ Toën introduced derived Morita theory.
In a nutshell, objects of derived Morita theory are (sheaves of) $\infty$-categories equipped with an action of the derived $\infty$-category $\QCoh(X)$.
However, these $\infty$-categories are not considered up to equivalence, but only up to \emph{Morita equivalence} (two $\QCoh(X)$-linear stable $\infty$-categories are Morita equivalent if there is a $\QCoh(X)$-linear equivalence between their categories of $\QCoh(X)$-valued presheaves, generalizing the classical Morita equivalence between rings, defined by looking at equivalences between the corresponding category of modules).
In this paper we denote by $\PrLomega_X$ the resulting $\infty$-category.
A key result of To\"en (later revisited by J.\ Lurie) is that $\PrLomega_X$ admits a symmetric monoidal structure.
Building on this theory, in \cite{Toen_Azumaya}, B.\ Toën proved that every cohomology class in $\rH^2\et(X;\bbG_m)$ can be represented by an invertible object in $\PrLomega_X$, thus completing the above table.
He actually went much further than that: replacing vector bundles by perfect complexes, he adapted the original definition of Grothendieck of sheaf of Azumaya algebras, following the general philosophy of derived geometry. He called the resulting class of objects \emph{derived Azumaya algebras} and proved that, when considered up to Morita equivalence, they coincide exactly with the invertible objects of $\PrLomega$.
The resulting group is nowadays called the \emph{derived Brauer group} and denoted $\mathrm{dBr}(X)$.
This group contains the classical Brauer group  $\rH^2\et(X;\bbG_m)$ of $X$ as a summand, but it has better formal properties.\\

The goal of this work is to exploit once more Toën's categorical framework in order to deduce new results about the Brauer-Grothendieck group even for classical schemes. We are interested in particular in problems related to \emph{descent questions} for  $\rH^2\et(X;\bbG_m)$ with respect to certain naturally occurring colimits, that we will solve by considering the corresponding problem for ``quasi-coherent'' sheaves of \emph{categories}.
We will consider two different kinds of situations:
\begin{enumerate}\itemsep=0.2cm
	\item the Beauville-Laszlo problem for $\rH^2\et(X;\bbG_m)$, and
	
	\item the formal GAGA problem for $\rH^2\et(X;\bbG_m)$,
\end{enumerate}
that we now describe in details.
\subsection{Beauville-Laszlo}
Our first result is about the categorification of the classical Beauville-Laszlo theorem \cite{Beauville-Laszlo}, which asserts that a quasi-coherent sheaf on an affine scheme $X$ can be obtained by patching a quasi-coherent sheaf on the open complement $U:=X-Z$ of a Cartier divisor $Z$ with a sheaf on the completion $\hat{X}_Z$, provided that they agree along the punctured tubular neighborhood $\hat{{X}}_Z-Z$. In fact, following \cite[1.3]{Bhatt_algebraization_2014}, we can consider more generally a pullback square
\[ \begin{tikzcd}
	V \arrow{r} \arrow{d} & U \arrow{d} \\
	T \arrow{r}{f} & S \ ,
\end{tikzcd} \]
where $S$ and $T$ are derived affine schemes and $U$ is the open complement of a closed subset of $S = \Spec(A)$ cut out by a finitely generated ideal $I \subseteq \pi_0(A)$.

Assume that  $f$ induces an equivalence between $I$-completions. Then we have:

\begin{thm}[See \cref{thm:Beauville_Laszlo_omega}]\label{thm-intro:Beauville_Laszlo}
	The canonical map
	\[ \PrLomega_S \longrightarrow \PrLomega_U \times_{\PrLomega_V} \PrLomega_T \]
	is a symmetric monoidal equivalence.
	In particular, the map
	\[ \dAz_S \longrightarrow \dAz_U \times_{\dAz_V} \dAz_T \]
	is an equivalence as well.
	We thus obtain a long exact sequence
	\begin{center}
		\begin{tikzpicture}[descr/.style={fill=white,inner sep=1.5pt}]
			\matrix (m) [
			matrix of math nodes,
			row sep=1em,
			column sep=2.5em,
			text height=1.5ex, text depth=0.25ex
			]
			{ 0 & \cO(S)^\times & \cO(U)^\times \oplus \cO(T)^\times & \cO(V)^\times \\
				& \mathrm{dPic}(S) & \mathrm{dPic}(U) \oplus \mathrm{dPic}(T) & \mathrm{dPic}(V) \\
				& \dBr(S) & \dBr(U) \oplus \dBr(T) & \dBr(V) \\
			};
			
			\path[overlay,->, font=\scriptsize,>=latex]
			(m-1-1) edge (m-1-2)
			(m-1-2) edge (m-1-3)
			(m-1-3) edge (m-1-4)
			(m-1-4) edge[out=355,in=175] node[descr,yshift=0.3ex] {$\delta_{0}$} (m-2-2)
			(m-2-2) edge (m-2-3)
			(m-2-3) edge (m-2-4)
			(m-2-4) edge[out=355,in=175] node[descr, yshift=0.3ex]{$\delta_1$} (m-3-2)
			(m-3-2) edge (m-3-3)
			(m-3-3) edge (m-3-4);
		\end{tikzpicture}
	\end{center}
\end{thm}

In the above statement we wrote $\mathrm{dPic}(X)$ to denote the group of invertible objects in the stable $\infty$-category $\QCoh(X)$.
In other words,
\[ \mathrm{dPic}(X) \simeq \Pic(X) \oplus \rH^0\et(X;\Z) \ , \]
where the second component classifies the shifted structure sheaf $\cO_X[n]$ of $X$.

\medskip

Notice that the stack $\rK(\bbG_m,2)$ is \emph{not} Tannakian.
Therefore, one cannot simply deduce \cref{thm-intro:Beauville_Laszlo} from the Beauville-Laszlo theorem for perfect complexes (proven for instance in \cite[Theorem 7.4.0.1]{Lurie_SAG}).
One should rather think of \cref{thm-intro:Beauville_Laszlo} as an advance in the theory of descent, that goes beyond all statements that can be formally deduced from $\bfPerf$.
This theorem is akin to Toën's \cite[Theorem 0.2]{Toen_Azumaya}, that asserts the fpqc-local nature of compact generators, but relies also on the general theory of complete, local and nilpotent modules developed by Lurie in \cite[\S7]{Lurie_SAG}.

\subsection{Formal GAGA}
The second kind of colimit that we consider arises from the following geometric situation: let $S = \Spec(R)$ be the spectrum of a commutative ring which is complete with respect to a finitely generated ideal $I \subseteq \pi_0(R)$, and let $p \colon X \to S$ be a proper $S$-scheme.
For $i \geqslant 0$, set $S_i \coloneqq \Spec(R / I^{i+1})$ and $X_i \coloneqq S_i \times_S^{\mathrm d} X$. Let $\mathfrak{X}$ denote the formal completion of $X$ along $S_0 \times_S X$, i.e. $\mathfrak{X}=\colim_i X_i$.
For every positive integer $n$, we can consider the comparison map
\begin{equation}\label{eq:GAGA_comparison_map}
	\rH^n\et(X;\bbG_m) \longrightarrow \lim_i \rH^n\et(X_i;\bbG_m) \ .
\end{equation}
When $n = 0$ or $n = 1$, this map is an equivalence.
This is the consequence of Grothendieck's existence theorem and the equivalence $\rH^1\et(X;\bbG_m) \simeq \Pic(X)$.
However, already for $n = 2$, the situation is more complex.

\medskip

Indeed, in \cite[Exposé VI]{Grothendieck_Dix_expose} Grothendieck himself studied this problem.
He proved injectivity of \eqref{eq:GAGA_comparison_map} under the following two key assumptions:
\begin{itemize}\itemsep=0.2cm
	\item $X$ is regular and the map $p \colon X \to S$ is flat;
	
	\item the vanishing $\lim^1_i \Pic(X_i) \simeq 0$ holds;
\end{itemize}
This result can be improved using \emph{continuous cohomology} \cite{Jannsen_Continuous}, which is the \'etale cohomology $\rH^n\et(\fX;\bbG_m)$ and which can be explicitly defined as
\[ \rH^n\et(\fX,\bbG_m) \simeq \rH^n\Big( \lim_i \rR \Gamma\et(X_i, \bbG_m) \Big) , \]
where the limit inside the parentheses is understood in the $\infty$-categorical way (i.e.\ it is a homotopy limit).
Then \eqref{eq:GAGA_comparison_map} can be factored as
\[ \rH^2\et(X;\bbG_m) \longrightarrow \rH^2\et(\fX;\bbG_m) \longrightarrow \lim_i \rH^2\et(X_i;\bbG_m) \ . \]
In \cite[Theorem 7.2]{Geisser_Morin_Kernel_Brauer_Manin} it was shown that the first map is injective when $X$ is regular and flat over $\Z_p$.

\medskip

To\"en's framework of derived Morita theory allows to treat this problem in a much greater generality.
Representing classes in $\mathrm{dBr}(X)$ as invertible objects in the Morita theory $\PrLomega_X$ of $X$, one is naturally led to consider the natural comparison map
\[ \PrLomega_X \longrightarrow \lim_i \PrLomega_{X_i} . \]
Full faithfulness of this functor would immediately imply the injectivity of $\mathrm{dBr}(X) \to \mathrm{dBr}(\fX)$.
Although this is too much to hope for (just as one cannot expect Grothendieck's existence theorem to hold for arbitrary quasi-coherent sheaves), we can restrict to compact objects in $\PrLomega_X$.
These can be understood as sheaves of \emph{smooth and proper} stable $\infty$-categories with a coherent action of $\QCoh(X)$.
Denoting by $\mathsf{SmPr}^{\mathsf{cat}}(X)$ (resp.\ $\dAz^{\mathrm{cat}}(X)$) the full subcategory of $\PrLomega_X$ spanned by smooth and proper (resp.\ invertible) categories, we can now state our second main theorem:

\begin{thm}[{See \cref{thm:GAGA_Morita} and \cref{cor:injectivity_derived_Brauer_complete_base}}] \label{thm:GAGA_Introduction}
	Let $S = \Spec(R)$ be the spectrum of a Noetherian ring which is complete with respect an ideal $I$.
	Let $X \to S$ be a proper $S$-scheme.
	Then:
	\begin{enumerate}\itemsep=0.2cm
		\item The natural symmetric monoidal functor
		\begin{equation} \label{eq:map_Brauer_intro_categorification}
			\NcSmPr(X) \longrightarrow \NcSmPr(\fX) \coloneqq \lim_{i \in \bbN} \NcSmPr(X_i) .
		\end{equation}
		is fully faithful. In particular it induces a fully faithful functor on the subcategory of invertible objects 
		\[ \dAz^{\mathrm{cat}}(X) \longrightarrow \dAz^{\mathrm{cat}}(\mathfrak X) \coloneqq \lim_{i \in \bbN} \dAz^{\mathrm{cat}}(X_i) .\]
		
		\item \emph{(Formal injectivity)} The natural map
		\[ \dBr(X) \longrightarrow \dBr(\mathfrak X) \]
		is injective. In particular, the  group homomorphism $\rH^2\et(X, \bbG_m) \longrightarrow \rH^2\et(\mathfrak{X}, \bbG_m)$ is injective.
		
		\item Asssume that $R$ is a noetherian complete local ring. Let $\ell$ be a prime different from the residue characteristic.
		Then there exists a short exact sequence
		\[ 0 \longrightarrow \Pic(X) / \ell \longrightarrow \lim_i (\Pic(X_i) / \ell) \stackrel{\rho}{\longrightarrow} \lim_i^1 \Pic(X_i) \]
		such that
		\[ \mathrm{Im}(\rho) \simeq \ker\big( \mathrm{Br}(X)[\ell] \to \lim_i \mathrm{Br}(X_i) \big) \ . \]
	\end{enumerate}
	The conclusion of (2) holds if $S$ is assumed to be the spectrum of a local Noetherian Henselian ring $R$ such that the canonical map $R \to\lim_n R / \mathfrak m^n$  has regular geometric fibers. 
\end{thm}

Note that if $p\colon X\to S$ was assumed to be \emph{smooth}, full faithfulness of the functor \eqref{eq:map_Brauer_intro_categorification} would be a formal consequence of Lurie's \cite[Theorems 11.1.4.1, 11.3.6.1 and 11.4.4.1]{Lurie_SAG}.
In any case, the functor \eqref{eq:map_Brauer_intro_categorification} is not essentially surjective, even if we restrict ourselves to the subcategory of invertible objects (see \cref{rem:failure_surjectivity}).
It would be interesting to determine a specific class of \emph{algebraizable} formal families of smooth and proper (or even invertible) $\infty$-categories, i.e.\ an explicit characterization of the essential image of the functor \eqref{eq:map_Brauer_intro_categorification}.\\

\subsection{Grothendieck'es existence theorem for $\bbG_m$-gerbes}

We now offer another consequence of \cref{thm:GAGA_Introduction}-(1).
The elements of the Brauer-Grothendieck group have another, different, geometric interpretation.
Indeed, the work of Giraud \cite{Giraud_Cohomologie_1971} shows that a class $\alpha \in \rH^2\et(X,\bbG_m)$ can be represented by a $\bbG_m$-gerbe.
Informally, this is the datum of a (derived) stack $\pi \colon \mathfrak A \to X$ which étale locally looks like $\rB\bbG_m \times X$.
One might wonder whether it is possible to prove the formal injectivity statement of \cref{thm:GAGA_Introduction} from this point of view.
It turns out that in order to do so, one would need to know Grothendieck's existence theorem for arbitrary $\bbG_m$-gerbes (see the explanation at the beginning of \cref{sec:G-gerbes}).
However, to the best of our knowledge this is only known for $\bbG_m$-gerbes that have the resolution property \cite[Corollary 1.7]{Alper_Hall_Rydh_Etale}.
On the other hand, Totaro showed in \cite[Theorem 1.1]{Totaro_Resolution_property} that this is equivalent to ask that the $\bbG_m$-gerbe in question is a global quotient, and by \cite[Corollary 3.8 \& Example 3.12]{Vistoli_Brauer_quotient_stack} this is not always the case (since a $\bbG_m$-gerbe is a global quotient if and only if the class in $\rH^2\et(X,\bbG_m)$ is representable by a \emph{classical} Azumaya algebra).
\emph{Reversing the perspective}, we can use \cref{thm:GAGA_Introduction} to prove that Grothendieck's existence theorem holds for \emph{arbitrary} $\bbG_m$-gerbes:

\begin{thm}[{See \cref{cor:GAGA_gerbes}}]
	Let $S = \Spec(R)$ be the spectrum of a Noetherian ring which is complete with respect an ideal $I$.
	Let $X \to S$ be a proper derived $S$-scheme.
	Let $(\mathfrak A, \alpha)$ be a $\bbG_m$-gerbe over $X$.
	Then the canonical map
	\[ \Perf(\mathfrak A) \longrightarrow \lim_n \Perf(\mathfrak A_n) \]
	is a symmetric monoidal equivalence of stable, symmetric monoidal $\infty$-categories.
\end{thm}

\begin{rem}
	In the main body of the text we prove a finer result: notably, the same result holds true provided that $X$ is a quasi-compact geometric\footnote{See the discussion preceding \cref{thm:qcoh_gerbe} and in particular \cref{rem:geometric_stacks} for the definition of this notion. It includes all algebraic stacks in the sense of \cite[\href{https://stacks.math.columbia.edu/tag/026O}{Tag 026O}]{stacks-project}.} derived stack for which the canonical map
	\[ \Perf(X) \longrightarrow \lim_n \Perf(X_n) \]
	is an equivalence (i.e.\ $X$ is \emph{categorically proper} in the sense of \cite{Halpern-Leistner_Preygel}).
	It is worth observing that, after a first version of this paper was written, a proof of formal injectivity for $\bbG_m$-gerbes over separated algebraic spaces appeared in \cite{kresh-mathur}.
	The proof supplied there is more elementary, and relies on the Tannakian reconstruction theorem of Hall and Rydh \cite{Hall_Rydh_Coherent}, but it also yields less information.
	In particular, it seems to us that the method of \cite{kresh-mathur} does not lead to an alternative proof of the above theorem, especially in the strong formulation just described.
\end{rem}

\medskip
\paragraph{\textbf{Acknowledgments.}}
We would like to thank K\k{e}stutis \v{C}esnavi\v{c}ius, Jean-Louis Colliot-Th\'el\`ene, Andrea Di Lorenzo, Andrea Gagna, Thomas Geisser, Rune Haugseng, Lorenzo Mantovani, Guglielmo Nocera, Michele Pernice, Marco Robalo, Bertrand Toën, Shuji Saito, Gabriele Vezzosi and Angelo Vistoli for very interesting conversations on the subject of this paper.
F. B. is supported by the PRIN “Geometric, Algebraic and Analytic Methods in Arithmetic”.

\section{Review of derived Morita theory}

\subsection{Tensor products of presentable $\infty$-categories}

We start by recalling the basics concerning quasi-coherent sheaves of $\infty$-categories.

\begin{notation}
	\begin{itemize}\itemsep=0.2cm
		\item We let $\PrL$ denote the $\infty$-category of presentable $\infty$-categories and left adjoint functors between them, see \cite[Definition 5.5.3.1]{HTT}.
		
		\item We let $\PrLcg$ denote the \emph{full} subcategory of $\PrL$ spanned by compactly generated presentable $\infty$-categories.
		
		\item We let $\PrLomega$ denote the subcategory of $\PrLcg$ having all the objects and whose morphisms preserve compact objects.
	\end{itemize}
\end{notation}

The results of \cite[\S4.8]{Lurie_Higher_algebra} endow these $\infty$-categories with symmetric monoidal structures.
Given $\cC, \cD \in \PrL$ we denote by $\cC \otimes \cD$ their tensor product, characterized by the property that for every other presentable $\infty$-category $\cE$, $\FunL(\cC \otimes \cD, \cE)$ consists exactly of the \emph{bicontinuous} functors $\cC\times \cD \to \cE$.

\medskip

\begin{thm}[To\"en-Lurie] \label{thm:tensor_product_presentable}
	\hfill
	\begin{enumerate}\itemsep=0.2cm
		\item Let $\cC$ and $\cD$ be two presentable $\infty$-categories.
		If $\cC$ and $\cD$ are compactly generated, the same goes for $\cC \otimes \cD$.
		
		\item \label{thm:tensor_product_presentable:symmetric_monoidal} The tensor product $\otimes$ can be promoted to a symmetric monoidal structure on both $\PrL$ and $\PrLomega$.
		The natural inclusion functor $j \colon \PrLomega \to \PrL$ can be upgraded to a symmetric monoidal $\infty$-functor.
		
		\item \label{thm:tensor_product_presentable:closed} Let $\cC$ and $\cD$ be presentable $\infty$-categories.
		The evaluation map
		\[ \cC \times \FunL( \cC, \cD ) \longrightarrow \cD \]
		is bicontinuous and the induced map
		\[ \cC \otimes \FunL( \cC, \cD ) \longrightarrow \cD \]
		is a counit for the adjunction $- \otimes \cC \dashv \FunL( \cC, - )$.
		In particular, the symmetric monoidal structure on $\PrL$ is closed and $\otimes$ commutes with colimits in both variables.
	\end{enumerate}	
\end{thm}

\begin{proof}
	The first statement follows from \cite[Lemma 5.3.2.11-(2)]{Lurie_Higher_algebra}.
	Point (\ref{thm:tensor_product_presentable:symmetric_monoidal}) is exactly the content of \cite[Proposition 4.8.1.15]{Lurie_Higher_algebra}.
	Statement (\ref{thm:tensor_product_presentable:closed}) is discussed in \cite[Remark 4.8.1.18]{Lurie_Higher_algebra}.
\end{proof}

\begin{defin} \label{def:rigid_symmetric_monoidal}
	Let $R$ be a connective $\mathbb E_\infty$-ring and let $\cA$ be a presentably symmetric monoidal $R$-linear $\infty$-category.
	We say that $\cA$ is \emph{rigid} if it satisfies the following conditions:
	\begin{enumerate}\itemsep=0.2cm
		\item the right adjoint to the multiplication map $\cA \otimes \cA \to \cA$ is a morphism in $\PrL_R$;
		\item the right adjoint to the unit $\Mod_R \to \cA$ is a morphism in $\PrL_R$;
		\item every compact object of $\cA$ admits both a left and a right dual.
	\end{enumerate}
\end{defin}

The following technical result will be needed later on:

\begin{prop} \label{prop:compactly_generated_algebras}
	Let $R$ be a connective $\mathbb E_\infty$-ring and let $\cA$ be a presentably symmetric monoidal $R$-linear $\infty$-category.
	Assume that $\cA$ is compactly generated and rigid.
	Then:
	\begin{enumerate}\itemsep=0.2cm
		\item $\cA$ is a commutative algebra in $\PrLomega_R$;
		
		\item The canonical functor
		\[ \Mod_\cA( \PrLomega_R ) \longrightarrow \Mod_\cA( \PrL_R ) \]
		can be upgraded to a symmetric monoidal functor.
	\end{enumerate}
\end{prop}

\begin{proof}
	Consider the natural inclusion $j \colon \PrLomega_R \to \PrL_R$ equipped with the (natural) symmetric monoidal structure provided by \cref{thm:tensor_product_presentable}-(2).
	We can review it as a morphism of coCartesian fibrations
	\[ \begin{tikzcd}[column sep = small]
		(\PrLomega_R)^\otimes \arrow{rr}{j} \arrow{dr} & & (\PrL_R)^\otimes \arrow{dl} \\
		{} & \mathsf{Fin}_* & \phantom{(\PrL_R)^\otimes}.
	\end{tikzcd} \]
	Since $(\PrL_R)^\otimes$ and $(\PrLomega_R)^\otimes$ are $\infty$-operads, for every $\langle n \rangle \in \mathsf{Fin}_*$, we have canonical identifications
	\[ (\PrLomega_R)^\otimes_{\langle n \rangle} \simeq (\PrLomega_R)^{\times n}, \qquad (\PrL_R)^\otimes_{\langle n \rangle} \simeq (\PrL_R)^{\times n} . \]
	For every $1 \le i \le n$, let $\rho^n_i \colon \langle n \rangle \to \langle 1 \rangle$ be the inert morphism sending everything except $i \in \langle n \rangle$ to $* \in \langle 1 \rangle$.
	If $\cB \in (\PrL_R)^\otimes_{\langle n \rangle}$, we let
	\[ \cB \longrightarrow \cB_i \]
	be a coCartesian lift of $\rho^n_i$.
	The objects $\cB_i$ are well defined up to a contractible space of choices, and we therefore allow ourselves to write
	\[ \cB \simeq (\cB_i)_{i \in \langle n \rangle^\circ} . \]
	Since $j$ is a symmetric monoidal functor (i.e.\ it preserves coCartesian edges), we see that a coCartesian morphism $\cB \to \cC$ in $(\PrL_R)^\otimes$ belongs to $(\PrLomega_R)^\otimes$ if and only if $\cB$ does.\\
	
	We can identify $\CAlg(\PrL_R)$ with the category of maps of $\infty$-operads $\mathsf{Fin}_* \to (\PrL_R)^\otimes$, and similarly for $\CAlg(\PrLomega_R)$.
	Since the functor $j$ is faithful, we see that a section of $s \colon \mathsf{Fin}_* \to (\PrL_R)^\otimes$ factors through $(\PrLomega_R)^\otimes$ if and only if it satisfies the following two conditions:
	\begin{enumerate}\itemsep=0.2cm
		\item for every object $\langle n \rangle \in \mathsf{Fin}_*$, $s(\langle n \rangle ) \in (\PrLomega_R)^\otimes$;
		\item for every morphism $f \colon \langle n \rangle \to \langle m \rangle$, $s(f)$ is a morphism in $(\PrLomega_R)^\otimes$.
	\end{enumerate} 
	Let $s \colon \mathsf{Fin}_* \to (\PrL_R)^\otimes$ be the section corresponding to the presentably symmetric monoidal $\infty$-category $\cA$.
	We have to prove that rigidity of $\cA$ implies that $s$ satisfies the above two conditions.
	Since $s$ is a map of $\infty$-operads, we have canonical equivalences
	\[ s(\langle n \rangle)_i \simeq \cA . \]
	Since $\cA$ is compactly generated by assumption, we conclude that condition (1) is automatically satisfied.
	We now verify condition (2).
	Since every morphism in $\mathsf{Fin}_*$ can be written as the composition of an inert and an active morphism (see \cite[Remark 2.1.2.2]{Lurie_Higher_algebra}), we see that it is enough to prove that condition (2) is satisfied separately by these two classes of morphisms.
	To begin with, assume that $f$ is inert.
	Since $s$ is a map of $\infty$-operads, $s(f)$ is coCartesian.
	Since $s(\langle n \rangle)$ belongs to $(\PrLomega_R)^\otimes_{\langle n \rangle}$, we deduce that $s(f)$ belongs to $(\PrLomega_R)^\otimes$ as well.\\
	
	Assume now that $f$ is active.
	In this case, we can factor $f$ as a surjective morphism followed by an injective one, and it is therefore enough to prove that condition (2) holds separately in these two cases.
	Assume first that $f$ is active and surjective.
	Then $f$ can be (non-uniquely) factored as a composition
	\[ \langle n \rangle \xrightarrow{f_1} \langle n-1\rangle \xrightarrow{f_2} \cdots \xrightarrow{f_m} \langle m \rangle , \]
	where each $f_i$ is active and surjective.
	It therefore enough to treat the case of a morphism $g \colon \langle n' \rangle \to \langle n' - 1 \rangle$ which is active and surjective.
	In this case, there exists an element $k \in \langle n'-1 \rangle^\circ$ such that:
	\begin{enumerate}\itemsep=0.2cm
		\item $g\inv(k) = \{i_1, i_2\}$ consists of exactly two elements;
		\item every other element of $\langle n' -1 \rangle$ has exactly one preimage via $g$.
	\end{enumerate}
	Let
	\[ s(\langle n' \rangle) \longrightarrow \cB \]
	be a coCartesian lift for $g$ starting at $s(\langle n' \rangle)$.
	Then $s(f)$ belongs to $(\PrLomega_R)^\otimes$ if and only if the uniquely determined morphism $\gamma \colon \cB \to s(\langle n'-1 \rangle)$ does.
	Unraveling the definitions, we have:
	\[ \cB_h \simeq \begin{cases}
		s(\langle n' \rangle)_{g\inv(h)} & \text{if } h \ne k \\
		s(\langle n' \rangle)_{i_1} \otimes s(\langle n' \rangle)_{i_2} & \text{if } h = k.
	\end{cases} \]
	In other words, $\cB_h \simeq \cA$ if $h \ne k$ and $\cB_k \simeq \cA \otimes \cA$.
	With these conventions, we see that $\gamma_h$ is (equivalent to) the identity of $\cA$ if $h \ne k$ and it is the multiplication
	\[ \cA \otimes \cA \longrightarrow \cA \]
	when $h = k$.
	Since this map belongs to $\PrLomega_R$ by assumption, the conclusion follows in the active and surjective case.\\
	
	We are left to deal with the case where $f \colon \langle n \rangle \to \langle m \rangle$ is active and injective.
	Let once again
	\[ s(\langle n \rangle) \longrightarrow \cB \]
	be the coCartesian lift of $f$ starting at $s(\langle n \rangle)$.
	Then $s(f)$ belongs to $(\PrLomega_R)^\otimes$ if and only if the uniquely determined morphism $\delta \colon \cB \to s(\langle m \rangle)$ does.
	Unraveling the definitions, we see that
	\[ \cB_k = \begin{cases}
		\cA & \text{if } k \in \mathrm{Im}(f) \\
		\Mod_R & \text{otherwise}.
	\end{cases} \]
	In these terms, the map $\delta_k$ is the identity of $\cA$ if $k \in \mathrm{Im}(f)$, and it corresponds to the the map $\Mod_R \to \cA$ selecting the tensor unit of $\cA$ otherwise.
	As the latter is in $\PrLomega_R$ by assumption, the proof of point (1) follows.\\
	
	As for statement (2), it follows from the fact that $\PrLomega_R \to \PrL_R$ is strong monoidal and commutes with small colimits together with \cite[Theorem 4.2.2.8]{Lurie_Higher_algebra}.
\end{proof}

\subsection{Smooth and proper categories}

Let $X$ be a quasi-compact and quasi-separated derived scheme.
Bondal-Van den Bergh's theorem \cite{Bondal_VdB} implies that $\QCoh(X)$ is compactly generated, and therefore defines an object in $\PrLomega$.
More it is true: since $\QCoh(X) \simeq \Ind(\Perf(X))$, and perfect complexes are exactly dualizable objects in $\QCoh(X)$, we see that $\QCoh(X)$ is a \emph{rigid} symmetric monoidal $\infty$-category, in the sense of Definition~\ref{def:rigid_symmetric_monoidal}.
Thus, Proposition~\ref{prop:compactly_generated_algebras} implies that $\QCoh(X)$ defines an algebra in $\PrLomega$.
We can therefore give the following definition:

\begin{defin}
	Let $X$ be a qcqs derived scheme.
	We define:
	\begin{enumerate} \itemsep=0.2cm
		\item the \emph{$\infty$-category of $\QCoh(X)$-linear presentable $\infty$-categories} $\PrL_X$ as
		\[ \PrL_X\coloneqq \Mod_{\QCoh(X)}(\PrL) . \]
		
		\item The \emph{$\infty$-category of compactly generated $\QCoh(X)$-linear $\infty$-categories} $\PrLcg_X$ as
		\[ \PrLcg_X \coloneqq \PrL_X \times_{\PrL} \PrLcg . \]
		
		\item The \emph{$\infty$-category of compactly generated $\QCoh(X)$-linear $\infty$-categories up to Morita equivalence} $\PrLomega_X$ (or simply the \emph{Morita theory of $X$}) as
		\[ \PrLomega_X \coloneqq \Mod_{\QCoh(X)}(\PrLomega) . \]
	\end{enumerate}
\end{defin}

Let us briefly review the descent properties of these categories.
Let us observe the following mild generalization of \cite[5.4.5.7]{HTT}, whose proof is standard:

\begin{lem} \label{lem:compact_objects_finite_limit}
	Let $F \colon I \to \PrLomega$ be a finite diagram.
	For every $i \in I$, let $\cC_i \coloneqq F(i)$ and let
	\[ \cC \coloneqq \lim_{i \in I} \cC_i \]
	the limit being computed in $\PrL$.
	Let $f_i \colon \cC \to \cC_i$ be the natural projections.
	If $c \in \cC$ is an object and its components $f_i(c)$ are compact, the same goes for $c$.
\end{lem}

\ifpersonal
\begin{proof}
	Let $G \colon J \to \cC$ be a filtered diagram.
	For every $\alpha \in J$, let $d_\alpha \coloneqq G$ and set
	\[ d \coloneqq \colim_{\alpha \in J} d_\alpha . \]
	Since $I$ is finite and $J$ is filtered, we have
	\begin{align*}
		\Map_{\cC}( c, \colim_{\alpha \in J} d_\alpha ) & \simeq \lim_{i \in I} \Map_{\cC_i}( f_i(c), \colim_{\alpha \in J} f_i(d_\alpha) ) \\
		& \simeq \lim_{i \in I} \colim_{\alpha \in J} \Map_{\cC_i}( f_i(c), f_i(d_\alpha) ) \\
		& \simeq \colim_{\alpha \in J} \lim_{i \in I} \Map_{\cC_i}( f_i(c), f_i(d_\alpha) ) \\
		& \simeq \colim_{\alpha \in J} \Map_\cC( c, d_\alpha ) .
	\end{align*}
	This shows that $c$ is compact in $\cC$.
\end{proof}
\fi

Work of Toën, Lurie and Gaitsgory implies the following:

\begin{thm} \label{thm:catQCoh_descent}
	\hfill
	\begin{enumerate}\itemsep=0.2cm
		\item The assignments $X \mapsto \PrL_X$ and $X \mapsto \PrLcg_X$ satisfy \'etale descent on qcqs derived schemes.
		
		\item The assignment $X \mapsto \PrLomega_X$ satisfies Zariski descent on qcqs derived schemes.
	\end{enumerate}
\end{thm}

\begin{proof}
	The functor
	\[ \dAff\op \longrightarrow \Cat_\infty \]
	sending $S$ to $\PrL_S$ satisfies \'etale descent as a consequence of \cite[Theorem D.3.6.2 \& Corollary D.3.3.5]{Lurie_SAG}.
	Theorem D.5.3.1 in \emph{loc.\ cit.} shows that the subfunctor given by the rule $S \mapsto \PrLcg_S$ also satisfies \'etale descent.
	Let now $X$ be a (not necessarily affine) scheme and let $U_\bullet$ be an \'etale affine hypercover of $X$.
	It follows from \cite[Theorem 2.1.1]{Gaitsgory_1_affineness} that the canonical map
	\[ \PrL_X \longrightarrow \lim_{[n] \in \mathbf \Delta} \PrL_{U_n} \]
	is an equivalence, whence the conclusion.
	Furthermore, using \cite[Theorem 0.2]{Toen_Azumaya}, we see that the same statement holds for $\PrLcg_X$.
	
	\medskip
	
	We now turn to statement (2).
	Proceeding by induction on the number of opens forming a cover of $X$, we see that it is enough to prove the following statement: if $U$ and $V$ are two open Zariski subsets of $X$, then the canonical map
	\[ \PrLomega_X \longrightarrow \PrLomega_U \times_{\PrLomega_{U\cap V}} \PrLomega_V \]
	is an equivalence.
	Consider the following diagram
	\[ \begin{tikzcd}
		\PrLomega_X \arrow{r} \arrow{d} & \PrLomega_U \times_{\PrLomega_{U \cap V}} \PrLomega_V \arrow{d} \\
		\PrLcg_X \arrow{r} & \PrLcg_U \times_{\PrLcg_{U \cap V}} \PrLcg_V \ .
	\end{tikzcd} \]
	The vertical arrows are faithful, and the bottom horizontal map is an equivalence by point (1).
	It follows that the top horizontal map is faithful as well, and \cref{lem:compact_objects_finite_limit} implies that it is fully faithful as well.
	Essential surjectivity follows from the fact that the bottom map is an equivalence and the fact that the functor $\PrLomega \hookrightarrow \PrLcg$ is symmetric monoidal and full on equivalences (and that the restrictions $\QCoh(X) \to \QCoh(U)$ and $\QCoh(X) \to \QCoh(V)$ commute with compact objects).
\end{proof}

\begin{defin}
	Let $X$ be a quasi-compact and quasi-separated derived scheme.
	We say that a compactly generated, presentable stable $\QCoh(X)$-linear $\infty$-category $\cC$ is \emph{smooth and proper} if it is a dualizable object in $\PrLomega_X$.
	We let $\NcSmPr(X)$ denote the full subcategory of $\PrLomega_X$ spanned by smooth and proper $\infty$-categories.
\end{defin}

\begin{rem}
	\hfill
	\begin{enumerate}\itemsep=0.2cm
		\item We refer the reader to \cite[\S11]{Lurie_SAG} for a separate analysis of the notion of smoothness and of properness for stable $\infty$-categories.
		
		\item We use the terminology smooth and proper to emphasize that dualizability holds in $\PrLomega_X$ and not just in $\PrL_X$.
		Indeed, every object in $\PrLomega_X$ is dualizable when seen inside $\PrL_X$, but it needs not to be dualizable inside $\PrLomega_X$ itself.
		
		\item Thanks to \cref{thm:catQCoh_descent}-(2), one can check dualizability in $\PrLomega_X$ Zariski-locally on $X$.
	\end{enumerate}
\end{rem}

\subsection{Derived Azumaya algebras}

Derived Azumaya algebras were introduced in \cite{Toen_Azumaya} as derived counterpart of classical Azumaya algebras.
When $X$ is a scheme, a classical Azumaya algebra is sheaf $\cA$ of $\cO_X$-algebras which is \'etale-locally isomorphic to $\cEnd(V)$ for some (locally defined) vector bundle $V$.
In the derived setting, we give the following definition:

\begin{defin}
	Let $X$ be a derived scheme.
	A \emph{derived Azumaya algebra over $X$} is an object $\cA \in \mathrm{Alg}(\Perf(X))$ satisfying the following two conditions:
	\begin{enumerate}\itemsep=0.2cm
		\item the natural map
		\[ \cA \otimes_{\cO_X} \cA\op \longrightarrow \cHom_{\cO_X}(\cA, \cA) \]
		is an equivalence;
		\item $\cA$ is a compact generator for $\QCoh(X)$.
	\end{enumerate}
\end{defin}

\begin{rem}
	In \cite[Definition 2.1]{Toen_Azumaya} the above definition is given only for affine derived schemes.
	The first condition can obviously be checked \'etale locally.
	As for the second one, \cite[Theorem 0.2]{Toen_Azumaya} shows that it can be checked Zariski locally.
	Therefore, a derived Azumaya algebra over $X$ is equivalently a perfect complex on $X$ equipped with an associative multiplication which Zariski locally is a derived Azumaya algebra in the sense of \cite{Toen_Azumaya}.
\end{rem}

A derived Azumaya algebra $\cA$ over a derived scheme $X$ gives automatically rise to an object $\cA\textrm{-}\Mod \in \PrLomega_X$.

\begin{defin}
	Let $X$ be a derived scheme.
	We let $\dAz^{\mathrm{cat}}(X)$ be the full subcategory of $\PrLomega_X$ spanned by quasi-coherent sheaves of categories of the form $\cA \textrm{-} \Mod$, where $\cA$ is a derived Azumaya algebra over $X$.
	We refer to $\dAz^{\mathrm{cat}}(X)$ as the \emph{$\infty$-category of derived Azumaya algebras on $X$ up to Morita equivalence}.
	We let $\dAz(X)$ denote the maximal $\infty$-groupoid contained inside $\dAz^{\mathrm{cat}}(X)$.
\end{defin}

\begin{rem} \label{rem:Morita_equivalence_dAz}
	\begin{enumerate}\itemsep=0.2cm
		\item Let $X$ be a derived scheme and let $\cA, \cA'$ be two derived Azumaya algebras over $X$.
		Then $\cA$ and $\cA'$ are Morita equivalent (i.e.\ they are equivalent when seen as objects in $\dAz^{\mathrm{cat}}(X)$) if and only if there is a $\QCoh(X)$-linear equivalence $\cA \textrm{-} \Mod \simeq \cA' \textrm{-} \Mod$.
		
		\item Using \cite[Corollary 2.1.4]{Lurie_Brauer}, we deduce that a derived Azumaya algebra $\cA$ over $X$ is Morita equivalent to $\cO_X$ if and only it is quasi-isomorphic to an algebra of the form $\cEnd(\cF)$ for some perfect complex $\cF \in \Perf(X)$.
		
		\item Let $X$ be a derived scheme. Using \cite[Proposition 2.14]{Toen_Azumaya}, one shows that \'etale locally on $X$ a derived Azumaya algebra on $X$ is always Morita equivalent to $\cO_X$.
	\end{enumerate}
\end{rem}

Let $X$ be a derived scheme and let $\cA$ be a derived Azumaya algebra on $X$.
Then $\cA\op$ is again a derived Azumaya algebra, and
\[ \cA \textrm{-}\Mod \otimes_{\QCoh(X)} \cA\op\textrm{-} \Mod \simeq \cEnd(\cA) \textrm{-} \Mod . \]
\Cref{rem:Morita_equivalence_dAz}-(2) implies that $\cA\textrm{-}\Mod \in \PrLomega_X$ is an invertible object.
A major theorem of \cite{Toen_Azumaya} states that the vice-versa is true:

\begin{thm}[{To\"en, \cite[Proposition 2.5]{Toen_Azumaya}}] \label{thm:Toen_invertible_objects}
	Let $X$ be a derived scheme.
	The full subcategory $\dAz^{\mathrm{cat}}(X) \subseteq \PrLomega_X$ coincides with the full subcategory of $\PrLomega_X$ spanned by invertible objects.
\end{thm}

\begin{cor}\label{cor:etale_descent_Azumaya}
	The functors
	\[ \dAz^{\mathrm{cat}} \colon \dAff\op \longrightarrow \Cat_\infty, \qquad \dAz \colon \dAff\op \longrightarrow \cS \]
	satisfy \'etale descent.
\end{cor}

\personal{We are hiding the theorem on gluing of compact generators. From this perspective, the theorem could be restated saying that the natural map from $\dAz(X)$ to the value of the sheaf $\dAz$ evaluated on $X$ is an equivalence.
	This is obvious for affine schemes, but not at all obvious for global ones.}

\begin{proof}
	We first observe that an object in $\PrLomega_X$ is invertible if and only if it is invertible in $\PrLcg_X$.
	The conclusion therefore follows from \cite[D.5.3.1]{Lurie_SAG}, asserting that the assignment $X \mapsto \PrLcg_X$ satisfies \'etale descent on $\dAff$.
\end{proof}

Observe that \cref{rem:Morita_equivalence_dAz}-(3) implies that the natural map $\Spec(\Z) \to \mathbf{dAz}$ is an effective epimorphism.
Studying the loop stack, To\"en deduces:

\begin{thm}[{To\"en, \cite[Theorem 3.12]{Toen_Azumaya}}] \label{thm:Toen_dAz}
	There is a natural equivalence of stacks
	\[ \dAz \simeq \rK(\bbG_m, 2) \times \rK(\Z, 1) . \]
\end{thm}

Motivated by the above theorem, we introduce the following notations:

\begin{defin}
	Let $F \in \dSt$ be a derived stack.
	The derived stack $\mathbf{dAz}(F)$ parametrizing families of derived Azumaya algebras on $F$ is defined as
	\[ \mathbf{dAz}(F) \coloneqq \bfMap(F, \dAz) . \]
\end{defin}

\begin{rem}
	Observe that the $\Spec(\mathbb Z)$-points of $\mathbf{dAz}(F)$ coincide with $\Map_{\dSt}(F, \dAz)$.
	We denote them simply by $\dAz(F)$.
\end{rem}

Finally, following \cite[Definition 3.14]{Toen_Azumaya} we introduce the derived Brauer group and its variants.
We start with the categorical derived Brauer group:

\begin{defin}
	Let $F \in \dSt$ be a derived stack.
	The \emph{categorical derived Brauer group of $F$} is the group $\dBr^{\mathrm{cat}}(F) \coloneqq \pi_0(\dAz(F))$.
\end{defin}

\begin{rem}
	Let $F$ be a derived stack.
	\begin{enumerate}\itemsep=0.2cm
		\item If $\cA$ is a derived Azumaya algebra on $F$ (defined globally \emph{up to quasi-isomorphism}\footnote{As opposed to \emph{up to Morita equivalence}.}), then it defines a class $[\cA] \in \dBr^{\mathrm{cat}}(F)$.
		
		\item The $\infty$-groupoid $\dAz(F)$ is in fact a $2$-homotopy type: write $\mathsf{dPic}(F)$ for the derived Picard groupoid of $F$ (that is, the maximal $\infty$-groupoid of invertible objects in $\QCoh(F)$).
		Then one has
		\[ \pi_1(\dAz(F)) \simeq \mathsf{dPic}(F) \ , \qquad \pi_2(\dAz(F)) \simeq \cO(F)^\times \ . \]
		Both homotopy groups are computed taking the trivial derived Azumaya algebra as basepoint.
	\end{enumerate}

\end{rem}

\begin{defin}
	Let $F \in \dSt$ be a derived stack.
	The \emph{derived Brauer group $\dBr(F)$ of $F$} is the subgroup of $\dBr^{\mathrm{cat}}(F)$ spanned by the classes of the form $[\cA]$ for $\cA$ a derived Azumaya algebra defined globally on $F$.
\end{defin}

\begin{warning}
	In this paper, $\dSt$ denotes the $\infty$-category of hypercomplete \emph{\'etale} sheaves on $\dAff$.
	It follows that the above derived Brauer groups are \emph{\'etale versions} of the derived Brauer groups considered in \cite{Toen_Azumaya}.
	The arguments given in loc.\ cit.\ apply verbatim for these \'etale derived Brauer groups.
\end{warning}

By definition, for every derived stack $F$ there is a natural inclusion $\dBr(F) \subseteq \dBr^{\mathrm{cat}}(F)$.
In \cite{Toen_Azumaya} it is shown that this map is surjective in many cases:

\begin{thm}[{To\"en, \cite[Corollary 4.8]{Toen_Azumaya}}] \label{thm:Toen_global_generation}
	Let $X$ be a quasi-compact and quasi-separated derived scheme.
	Then the natural map $\dBr(X) \subseteq \dBr^{\mathrm{cat}}(X)$ is an isomorphism.
\end{thm}

\begin{rem}
	Let $X$ be a quasi-compact and quasi-separated derived scheme.
	Then combining Theorems \ref{thm:Toen_dAz} and \ref{thm:Toen_global_generation} we obtain
	\[ \dBr(X) \simeq \pi_0 \Map_{\dSt}( X, \rK(\bbG_m,2) \times \rK(\Z, 1) ) \simeq \rH^2_{\mathrm{\'et}}(X, \bbG_m) \times \rH^1_{\mathrm{\'et}}(X, \Z) . \]
	In particular, every class $\alpha \in \rH^2_{\mathrm{\'et}}(X, \bbG_m)$ (even non-torsion ones!) can be realized as the class of some \emph{derived} Azumaya algebra over $X$.
\end{rem}

\section{Categorified Beauville-Laszlo theorem}

In this section we provide a categorification of the classical Beauville-Laszlo theorem.
Recall the basic setup: let $S = \Spec(A)$ and $T = \Spec(B)$ be two affine derived schemes.
Let $I \subseteq \pi_0(A)$ be a finitely generated ideal.
Let
\[ \begin{tikzcd}
	V \arrow{r}{g} \arrow{d}{j} & U \arrow{d}{i} \\
	T \arrow{r}{f} & S \ .
\end{tikzcd} \]
be a pullback diagram, where $U$ denotes the open complementary of the closed subset determined by $I$.
For the rest of this section, we will be working under the following assumption:

\begin{assumption}\label{assumption:Beauville_Laszlo}
	The map $f \colon T \to S$ induces an equivalence between the $I$-completions $A^\wedge_I \to B^\wedge_I$.
\end{assumption}

\noindent Then it is shown in \cite[Theorem 7.4.0.1]{Lurie_SAG} that the canonical map
\[ \QCoh(S) \longrightarrow \QCoh(T) \times_{\QCoh(V)} \QCoh(U) \]
is an equivalence of stable $\infty$-categories.
This is a far-reaching generalization of the classical Beauville-Laszlo theorem, originally formulated for vector bundles.
The main goal of this section is to prove the following \emph{categorified} version:

\begin{thm}\label{thm:Beauville_Laszlo_omega}
	The canonical restriction functor
	\begin{equation}\label{eq:categorified_Beauville_Laszlo}
		\PrLomega_S \longrightarrow \PrLomega_T \times_{\PrLomega_V} \PrLomega_U
	\end{equation}
	is an equivalence of symmetric monoidal $\infty$-categories.
\end{thm}

Passing to dualizable (resp.\ invertible) objects in the above equivalence, we immediately obtain:

\begin{cor}
	The equivalence \eqref{eq:categorified_Beauville_Laszlo} restricts to equivalences 
	\[ \NcSmPr_S \simeq \NcSmPr_T \times_{\NcSmPr_V} \NcSmPr_U \qquad \text{and} \qquad \dAz_S \simeq \dAz_T \times_{\dAz_V} \dAz_U \ . \]
	In particular, we obtain a long exact sequence
	\begin{center}
		\begin{tikzpicture}[descr/.style={fill=white,inner sep=1.5pt}]
			\matrix (m) [
			matrix of math nodes,
			row sep=1em,
			column sep=2.5em,
			text height=1.5ex, text depth=0.25ex
			]
			{ 0 & \cO(S)^\times & \cO(U)^\times \oplus \cO(T)^\times & \cO(V)^\times \\
				& \mathrm{dPic}(S) & \mathrm{dPic}(U) \oplus \mathrm{dPic}(T) & \mathrm{dPic}(V) \\
				& \dBr(S) & \dBr(U) \oplus \dBr(T) & \dBr(V) \\
			};
			
			\path[overlay,->, font=\scriptsize,>=latex]
			(m-1-1) edge (m-1-2)
			(m-1-2) edge (m-1-3)
			(m-1-3) edge (m-1-4)
			(m-1-4) edge[out=355,in=175] node[descr,yshift=0.3ex] {$\delta_{0}$} (m-2-2)
			(m-2-2) edge (m-2-3)
			(m-2-3) edge (m-2-4)
			(m-2-4) edge[out=355,in=175] node[descr, yshift=0.3ex]{$\delta_1$} (m-3-2)
			(m-3-2) edge (m-3-3)
			(m-3-3) edge (m-3-4);
		\end{tikzpicture}
	\end{center}
\end{cor}

\subsection{General setup}

Consider the commutative square
\[ \begin{tikzcd}
	\PrL_S \arrow{r}{f^\ast} \arrow{d}{i^\ast} & \PrL_T \arrow{d}{j^\ast} \\
	\PrL_U \arrow{r}{g^\ast} & \PrL_V \ .
\end{tikzcd} \]
This induces a canonical functor
\[ \Phi \colon \PrL_S \longrightarrow \PrL_T\times_{\PrL_V} \PrL_U \ , \]
which admits a right adjoint $\Psi$ that can be explicitly described as follows.
Let us represent an object in the fiber product $\PrL_T \times_{\PrL_V} \PrL_U$ as a quintuple $\underline{\cC} = (\cC_T, \cC_U, \cC_V, \alpha, \beta)$, where $\cC_T \in \PrL_T$, $\cC_U \in \PrL_U$, $\cC_V \in \PrL_V$ and 
\[ \alpha \colon \cC_T \longrightarrow \cC_V \qquad \text{(resp.\ } \beta \colon \cC_U \longrightarrow \cC_V \text{)} \]
is a $\QCoh(T)$-linear (resp.\ $\QCoh(U)$-linear) functor inducing an equivalence $\cC_T \otimes_{\QCoh(T)} \QCoh(V) \simeq \cC_V$ (resp.\ $\cC_U \otimes_{\QCoh(U)} \QCoh(V) \simeq \cC_V$).
Then $\Psi$ sends such a datum to the fiber product
\[ \begin{tikzcd}
	\Psi(\underline{\cC}) \arrow{r} \arrow{d} & f_\ast(\cC_T) \arrow{d}{\alpha} \\
	i_\ast(\cC_U) \arrow{r}{\beta} & i_\ast g_\ast(\cC_V) \ .
\end{tikzcd} \]
The functor $\Phi$ restricts to the functor \eqref{eq:categorified_Beauville_Laszlo}:
\[ \Phi^\omega \colon \PrLomega_S \longrightarrow \PrLomega_T \times_{\PrLomega_V} \PrLomega_U \ . \]
Notice that a priori $\Psi$ does not take $\PrLomega_T \times_{\PrLomega_V} \PrLomega_U$ to $\PrLomega_S$.
Nevertheless, a mild variation of \cite[Theorem 0.2]{Toen_Azumaya} will show that this is true (see \cref{prop:Thomason_trick}).

\subsection{Full faithfulness}

We start discussing the full faithfulness of $\Phi^\omega$.

\begin{lem}\label{lem:Beauville_Laszlo_full_faitfhulness}
	Assume that $\cC \in \PrL_S$ is dualizable.
	Then the unit transformation
	\[ \cC \longrightarrow \Psi(\Phi(\cC)) \]
	is an equivalence.
	In particular, $\Phi$ restricts to a fully faithful functor
	\[ \Phi \colon \PrLcg_S \longrightarrow \PrLcg_T \times_{\PrLcg_V} \PrLcg_U \ . \]
\end{lem}

\begin{proof}
	Unraveling the definitions, we have to check that for every $\cC\in \PrLcg_S$, the square
	\[ \begin{tikzcd}
		\cC \arrow{r} \arrow{d} & \cC \otimes_{\QCoh(S)} \QCoh(T) \arrow{d} \arrow{d} \\
		\cC \otimes_{\QCoh(S)} \QCoh(U) \arrow{r} & \cC \otimes_{\QCoh(S)} \QCoh(V) 
	\end{tikzcd} \]
	is a pullback square.
	Thanks to \cite[Theorem 7.4.0.1]{Lurie_SAG}, we know that the square
	\begin{equation}\label{eq:master_Beauville_Laszlo}
		\begin{tikzcd}
			\QCoh(S) \arrow{r}{f^\ast} \arrow{d}{i^\ast} & \QCoh(T) \arrow{d}{j^\ast} \\
			\QCoh(U) \arrow{r}{g^\ast} & \QCoh(V)
		\end{tikzcd}
	\end{equation}
	is a pullback square.
	Since $\cC$ is dualizable in $\PrL_S$, \eqref{eq:master_Beauville_Laszlo} remains a pullback after tensoring with $\cC$, whence the first half.
	As for the second statement, observe that $\QCoh(S)$ is a rigid and compactly generated symmetric monoidal $\infty$-category.
	Thus, \cite[Proposition D.5.4]{Gaitsgory_1_affineness} shows that if $\cC$ is compactly generated (hence dualizable in $\PrL$), it is dualizable in $\PrL_S$ as well.
	The conclusion follows.
\end{proof}

A little extra effort allows to pass from $\PrLcg$ to $\PrLomega$:

\begin{cor}
	The functor
	\[ \Phi^\omega \colon \PrLomega_S \longrightarrow \PrLomega_T \times_{\PrLomega_V} \PrLomega_U \]
	is fully faithful.
\end{cor}

\begin{proof}
	The canonical inclusion $\PrLomega_S \to \PrLcg_S$ is faithful.
	Thus, \cref{lem:Beauville_Laszlo_full_faitfhulness} immediately implies that the map
	\[ \Phi^\omega \colon \PrLomega_S \longrightarrow \PrLomega_T \times_{\PrLomega_V} \PrLomega_U \]
	is faithful as well.
	In order to check fullness, it is enough to check that given a functor
	\[ F \colon \cC \longrightarrow \cD \]
	between compactly generated $\QCoh(S)$-linear $\infty$-categories, if $\Phi(F)$ preserves compact objects, then so does $F$.
	We know from the proof of \cref{lem:Beauville_Laszlo_full_faitfhulness} that the square
	\[ \begin{tikzcd}
		\cD \arrow{r} \arrow{d} & f^\ast(\cD) \arrow{d} \\
		i^\ast(\cD) \arrow{r} & g^\ast i^\ast(\cD)
	\end{tikzcd} \]
	is a pullback.
	Moreover, the map $i^\ast(\cD) \to g^\ast i^\ast(\cD)$ is obtained by applying $\cD \otimes_{\QCoh(S)} -$ to the map $g^\ast \QCoh(U) \to \QCoh(V)$.
	Since the latter preserves compact objects and $\cD$ is compactly generated, the functoriality of the tensor product guarantees that $i^\ast(\cD) \to g^\ast i^\ast(\cD)$ preserves compact objects.
	A similar reasoning shows that $f^\ast(\cD) \to g^\ast i^\ast(\cD)$ preserves compact objects as well.
	Thus, the conclusion follows from \cref{lem:compact_objects_finite_limit}.
\end{proof}

\subsection{Essential surjectivity}

We start with a very general statement:

\begin{lem}\label{lem:essential_surjectivity}
	Let $\underline{\cC} = (\cC_T, \cC_U, \cC_V, \alpha, \beta)$ be an element of $\PrL_T \times_{\PrL_V} \PrL_U$.
	Then:
	\begin{enumerate}\itemsep=0.2cm
		\item the natural transformation
		\[ i^\ast \Psi(\underline{\cC}) \longrightarrow \cC_U \]
		is an equivalence;
		
		\item assume that $\cC_T$ is compactly generated.
		Then the natural transformation
		\[ f^\ast \Psi(\underline{\cC}) \longrightarrow \cC_T \]
		is an equivalence.
	\end{enumerate}
\end{lem}

\begin{proof}
	We start with statement (1).
	Since $\QCoh(S)$ is rigid and compactly generated, we see that $\QCoh(U)$ Is a dualizable object in $\PrLcg_S$.
	This implies that the square
	\[ \begin{tikzcd}
		i^\ast(\cC) \arrow{r} \arrow{d} & i^\ast f_\ast(\cC_T) \arrow{d} \\
		i^\ast i_\ast(\cC_U) \arrow{r} & i^\ast i_\ast g_\ast(\cC_V)
	\end{tikzcd} \]
	is a pullback.
	Moreover, since $U \to S$ is a Zariski open immersion, the canonical map
	\[ \QCoh(U) \otimes_{\QCoh(S)} \QCoh(U) \longrightarrow \QCoh(U) \]
	is an equivalence.
	This implies that for every $\cD \in \PrL_U$, the unit transformation $\cD \to i^\ast i_\ast(\cD)$ is an equivalence.
	Besides, using the base change equivalence $i^\ast f_\ast \simeq g_\ast j^\ast$ and the given identification $\cC_V \simeq j^\ast(\cC_T)$, we can rewrite the left square as
	\[ \begin{tikzcd}
		i^\ast(\cC) \arrow{r} \arrow{d} & g_\ast j^\ast(\cC_T) \arrow{d} \\
		\cC_U \arrow{r} & g_\ast j^\ast(\cC_T) \ .
	\end{tikzcd} \]
	Since the right vertical map is an equivalence and the square is a pullback, the same goes for the left vertical map.
	
	\medskip
	
	We now prove statement (2).
	As above, we see that $\QCoh(T)$ is a dualizable object in $\PrLcg_S$.
	Thus, the square
	\[ \begin{tikzcd}
		f^\ast(\cC) \arrow{r} \arrow{d} & f^\ast f_\ast(\cC_T) \arrow{d} \\
		f^\ast i_\ast(\cC_U) \arrow{r} & f^\ast f_\ast j_\ast(\cC_V)
	\end{tikzcd}  \]
	is a pullback.
	Using the base change equivalence $f^\ast i_\ast \simeq j_\ast g^\ast$, and the given identification $\cC_V \simeq g^\ast(\cC_U)$, we can rewrite it as
	\begin{equation}\label{eq:Beauville_Laszlo_essentially_surjective}
		\begin{tikzcd}
			f^\ast(\cC) \arrow{r} \arrow{d} & f^\ast f_\ast(\cC_T) \arrow{d} \\
			j_\ast j^\ast(\cC_T) \arrow{r} & f^\ast f_\ast j_\ast j^\ast(\cC_T) \ .
		\end{tikzcd}
	\end{equation}
	On the other hand, since $\QCoh(T)$ is dualizable as an object in $\PrLcg_S$, we see that the square \eqref{eq:master_Beauville_Laszlo} remains a pullback after tensoring with $\QCoh(T)$.
	In other words, the diagram
	\[ \begin{tikzcd}
		\QCoh(T) \arrow{r} \arrow{d} & \QCoh(T) \otimes_{\QCoh(S)} \QCoh(T) \arrow{d} \\
		\QCoh(U) \otimes_{\QCoh(S)} \QCoh(T) \arrow{r} & \QCoh(V) \otimes_{\QCoh(S)} \QCoh(T)
	\end{tikzcd} \]
	is a pullback square.
	Now, $\QCoh(T)$ is a rigid and compactly generated symmetric monoidal $\infty$-category.
	Since $\cC_T$ is compactly generated, we deduce that it is also a dualizable object in $\PrL_T$.
	Thus, the square
	\[ \begin{tikzcd}
		\cC_T \arrow{r} \arrow{d} & \QCoh(T) \otimes_{\QCoh(S)} \QCoh(T) \otimes_{\QCoh(T)} \cC_T \arrow{d} \\
		\QCoh(U) \otimes_{\QCoh(S)} \QCoh(T) \otimes_{\QCoh(T)} \cC_T \arrow{r} & \QCoh(V) \otimes_{\QCoh(S)} \QCoh(T) \otimes_{\QCoh(T)} \cC_T
	\end{tikzcd} \]
	obtained from the previous one applying $- \otimes_{\QCoh(T)} \cC_T$, is a pullback.
	Observe furthermore that we have a canonical identification
	\[ \QCoh(U) \otimes_{\QCoh(S)} \QCoh(T) \simeq \QCoh(V) \ . \]
	At this point, a simple diagram chase identifies the lower cospan of the above square with the one of \eqref{eq:Beauville_Laszlo_essentially_surjective}.
	Since both squares are pullback, we deduce that the canonical map
	\[ f^\ast(\cC) \longrightarrow \cC_T \]
	is an equivalence, whence the conclusion.
\end{proof}

To go any further, we need to study more in depth the functor $\Psi$.
More precisely, we need to know that it takes compactly generated gluing data to compactly generated $\infty$-categories.
For this, we need a couple of preliminaries.

\begin{lem}\label{lem:semiorthogonal_decomp}
	Let $T = \Spec(R)$ be an affine derived scheme and let $I \subseteq \pi_0(R)$ be a finitely generated ideal.
	Let $K$ be the closed subset of $T$ cut by $I$ and let $U \coloneqq T \smallsetminus K$.
	Let $\cD \in \PrL_T$ and set
	\[ \cD_K \coloneqq \QCoh_K(T) \otimes_{\QCoh(T)} \cD \qquad \text{and} \qquad \cD_V \coloneqq \QCoh(V) \otimes_{\QCoh(T)} \cD \ . \]
	Then:
	\begin{enumerate}\itemsep=0.2cm
		\item The functor $\cD_K \to \cD$ is fully faithful and cocontinuous;
		
		\item the functor $\cD \to \cD_V$ has a fully faithful right adjoint;
		
		\item the pair $(\cD_K, \cD_V)$ is a semiorthogonal decomposition of $\cD$;
	\end{enumerate}
	Assume additionally that $\cD$ is compactly generated.
	Then:
	\begin{enumerate}\itemsep=0.2cm	\setcounter{enumi}{3}
		\item $\cD_K$ is compactly generated, and $\cD_K^\omega = \cD_K \cap \cD^\omega$ (equivalently, the inclusion $\cD_K \hookrightarrow \cD$ preserves compact objects);
		
		\item assume that $\{F_i\}_{i \in I}$ is a set of compact objects of $\cD$ that belong to $\cD_K$ and form a system of generators for $\cD_K$, and that $\{E_j\}_{j \in J}$ is a set of compact objects of $\cD$ whose restriction to $\cD_V$ provides a system of compact generators for $\cD_V$.
		Then $\{F_i, E_j\}_{i \in I, j \in J}$ is a system of compact generators for $\cD$.
	\end{enumerate}
\end{lem}

\begin{proof}
	Combining \cite[Proposition 7.1.5.3 \& Corollary 7.1.2.11]{Lurie_SAG} we canonically identify the functor $\cD_K \to \cD$ with the inclusion $\cD^{\mathrm{Nil}(I)} \hookrightarrow \cD$.
	Moreover, \cite[Proposition 7.2.3.1]{Lurie_SAG} guarantees that $(\QCoh_K(T), \QCoh(V))$ is a semiorthogonal decomposition of $\QCoh(T)$.
	Therefore, the square
	\[ \begin{tikzcd}
		\QCoh_K(T) \arrow{r} \arrow{d} & \QCoh(T) \arrow{d} \\
		0 \arrow{r} & \QCoh(V) \ .
	\end{tikzcd} \]
	is both a pullback and a pushout.
	Applying $\cD \otimes_{\QCoh(T)} -$ the resulting square
	\[ \begin{tikzcd}
		\cD_K \arrow{r} \arrow{d} & \cD \arrow{d} \\
		0 \arrow{r} & \cD_V
	\end{tikzcd} \]
	is again a pushout.
	In particular, $\cD_V$ gets canonically identified with the localization of $\cD$ at $\cD_K$, and therefore the canonical functor $\cD \to \cD_V$ admits a fully faithful right adjoint.
	It immediately follows that $\cD_V$ is identified with $\cD^{\mathrm{Loc}(I)}$, and therefore that $(\cD_K, \cD_V)$ is a semiorthogonal decomposition of $\cD$.
	This proves statements (1), (2) and (3).
	
	\medskip
	
	Statement (4) follows from \cite[Proopsition 7.1.1.12-(e)]{Lurie_SAG}.
	Finally, (5) is a formal consequence of the other points: indeed, let $M \in \cD$ be an object and assume that for all $i \in I$ and $j \in J$, one has
	\[ \Map_\cD(F_i,M) \simeq \Map_{\cD}(E_j, M) \simeq 0 \ . \]
	Since the inclusion $\cD_K \hookrightarrow \cD$ commutes with colimits and the $F_i$ are compact generators for $\cD_K$, it follows from the first set of vanishing that $\Map_{\cD}(F,M) \simeq 0$ for every $F \in \cD_K$.
	Thus, $M \in \cD^{\mathrm{Loc}(I)}$ and at this point the second set of vanishing implies that $M \simeq 0$.
\end{proof}

\begin{prop} \label{prop:Thomason_trick}
	Let $\underline{\cC} = (\cC_T, \cC_U, \cC_V, \alpha, \beta)$ be an element of $\PrL_T \times_{\PrL_V} \PrL_U$.
	If $\cC_T$ and $\cC_U$ are compactly generated, then so is $\cC \coloneqq \cC_T \times_{\cC_V} \cC_U$.
\end{prop}

\begin{proof}
	Consider the commutative triangles
	\[ \begin{tikzcd}
		\cC_T \arrow{r} \arrow{dr}[swap]{\alpha} & \cC_T \otimes_{\QCoh(T)} \QCoh(V) \arrow{d} \\
		{} & \cC_V
	\end{tikzcd} \qquad \text{and} \qquad \begin{tikzcd}
		\cC_U \arrow{r} \arrow{dr}[swap]{\beta} & \cC_U \otimes_{\QCoh(U)} \QCoh(V) \arrow{d} \\
		{} & \cC_V
	\end{tikzcd} \]
	By definition, the vertical maps are equivalences, hence in particular they preserve compact objects.
	On the other hand, the functors
	\[ \QCoh(T) \longrightarrow \QCoh(V) \qquad \text{and} \qquad \QCoh(U) \longrightarrow \QCoh(V) \]
	preserve compact objects.
	Since $\cC_T$ and $\cC_U$ are compactly generated, it follows that the horizontal maps in the above two triangles also preserve compact objects.
	Thus, $\alpha$ and $\beta$ are automatically morphisms in $\PrLomega$.
	
	\medskip
	
	Let $K$ be the closed complementary of $V$ inside $T$.
	Then \cref{lem:semiorthogonal_decomp} guarantees that $(\cC_T)_K \coloneqq \QCoh_K(T) \otimes_{\QCoh(T)} \cD$ is a compactly generated full subcategory of $\cC_T$ and that the inclusion $(\cC_T)_K \subseteq \cC_T$ preserves compact objects.
	Choose a system of compact generators $\{F_i\}_{i \in I}$ for $(\cC_T)_K$.
	Their image in $\cC_V$ is zero, so every $F_i$ automatically defines an object $\widetilde{F}_i = (F_i, 0)$ in $\cC$, which is compact in virtue of \cref{lem:compact_objects_finite_limit}.
	
	\medskip
	
	Choose now a system of compact generators $\{G_j\}_{j \in J}$ for $\cC_U$.
	Set $E_j \coloneqq G_j \oplus G_j[1]$.
	Then $\{E_j\}_{j \in J}$ is again a system of compact generators of $\cC_U$.
	It follows that $\{g^\ast(E_j)\}_{j \in J}$ is a system of compact generators for $\cC_V$, and since
	\[ \begin{tikzcd}
		(\cC_T)_K \arrow{r} \arrow{d} & \cC_T \arrow{d} \\
		0 \arrow{r} & \cC_V
	\end{tikzcd} \]
	is both a pullback and a pushout square, Thomason's trick (\cite[Corollary 3.2.3]{Bondal_VdB}) implies that each $g^\ast(E_j) \simeq g^\ast(G_j) \oplus g^\ast(G_j)[1]$ is of the form $j^\ast(E_j')$ for some $E_j' \in \cC_T^\omega$.
	Therefore, the pair $(E_j', E_j)$ defines an object $\widetilde{E}_j \in \cC$, which is once again a compact object thanks to \cref{lem:compact_objects_finite_limit}.
	
	\medskip
	
	Consider now the family $\{(\widetilde{F}_i, \widetilde{E}_j)\}_{i \in I, j \in J}$.
	This is a system of compact objects of $\cC$.
	Its restriction to $\cC_U$ coincides with the family of compact generators $\{E_j\}_{j \in J}$, while its restriction to $\cC_T$ coincides with the system $\{F_i, E'_j\}_{i \in I, j \in J}$.
	Applying \cref{lem:semiorthogonal_decomp}-(5), we see that this is a system of compact generators of $\cC_T$, and therefore that $\{(\widetilde{F}_i, \widetilde{E}_j)\}_{i \in I, j \in J}$ is a system of compact generators for $\cC$.
	Thus, $\cC$ is compactly generated itself.
\end{proof}

\begin{cor}\label{cor:pullback_in_PrLomega}
	Let $\underline{\cC} = (\cC_T, \cC_U, \cC_V,\alpha, \beta)$ be an element of $\PrLcg_T \times_{\PrLcg_V} \PrLcg_U$.
	Then the fiber product
	\[ \begin{tikzcd}
		\cC \arrow{r} \arrow{d} & \cC_T \arrow{d}{\alpha} \\
		\cC_U \arrow{r}{\beta} & \cC_V
	\end{tikzcd} \]
	computed in $\PrL_S$, is equally a pullback in $\PrLcg_S$ and in $\PrLomega_S$.
\end{cor}

\begin{proof}
	Since $\PrLcg_S$ is fully faithful inside $\PrL_S$, the statement in this case is directly a consequence of \cref{prop:Thomason_trick}.
	For $\PrLomega_S$, observe that in this case \cref{lem:essential_surjectivity} implies that the canonical maps
	\[ \cC \otimes_{\QCoh(S)} \QCoh(U) \longrightarrow \cC_U \qquad \text{and} \qquad \cC \otimes_{\QCoh(S)} \QCoh(T) \longrightarrow \cC_T \]
	are equivalences.
	Thus, we can identify the map $\cC \to \cC_U$ (resp.\ $\cC \to \cC_T$) with the map obtained applying $\cC \otimes_{\QCoh(S)} -$ to $\QCoh(S) \to \QCoh(U)$ (resp.\ $\QCoh(S) \to \QCoh(T)$).
	As the latter commute with compact objects, the same goes for the original morphisms.
	Therefore the square in consideration is a square in $\PrLomega_S$.
	Since its image in $\PrLcg_S$ is a pullback, the conclusion immediately follows from \cref{lem:compact_objects_finite_limit}.
\end{proof}

We are now ready to prove \cref{thm:Beauville_Laszlo_omega}:

\begin{proof}[Proof of \cref{thm:Beauville_Laszlo_omega}]
	\Cref{cor:pullback_in_PrLomega} shows that the adjunction
	\[ \Phi \colon \PrL_S \leftrightarrows \PrL_T \times_{\PrL_V} \PrL_U \colon \Psi \]
	restricts to an adjunction
	\[ \Phi^\omega \colon \PrLomega_S \leftrightarrows \PrLomega_T \times_{\PrLomega_V} \PrLomega_U \colon \Psi^\omega \ . \]
	The unit of this adjunction is an equivalence thanks to \cref{lem:Beauville_Laszlo_full_faitfhulness}.
	On the other hand, \cref{lem:essential_surjectivity} guarantees that the counit is an equivalence as well.
	Thus $\Phi^\omega$ is an equivalence.
\end{proof}

\section{The formal GAGA problem for the derived Brauer group}

Let $S = \Spec(R)$ be the spectrum of an $\mathbb E_\infty$-ring which is complete with respect a finitely generated ideal $I \subseteq \pi_0(R)$.
Choose a tower $\{R_n\}_{n \geqslant 0}$ as in \cite[Lemma 8.1.2.2]{Lurie_SAG} computing the formal completion of $S$ along $I$ (when $R$ is discrete, we can simply take $R_n \coloneqq R / I^{n+1})$ 
Let $X$ be a derived stack over $\Spec(R)$ and let $X_n \coloneqq X \times_{\Spec(R)} \Spec(R_n)$.
The \emph{formal completion of $X$ along the base} is the formal scheme
\[ \mathfrak X \coloneqq \colim X_n . \]
We assume that $X$ is \emph{categorically proper} \cite{Halpern-Leistner_Preygel}, i.e.\ that the canonical symmetric monoidal functor
\[ \Perf(X) \longrightarrow \lim_n \Perf(X_n) \]
is an equivalence. This is satisfied for instance if $X$ is a proper derived algebraic space \cite[Theorem 8.5.0.3]{Lurie_SAG}.

\medskip

To this situation we can attach a number of cohomological invariants.
In first place, we have three different declinations of the derived Brauer group:
\[ \dBr(X), \qquad \dBr(\mathfrak X) , \qquad \lim_n \dBr(X_n) . \]
Similarly, we have three different declinations of the cohomological Brauer group:
\[ \rH^2_{\mathrm{\'et}}(X; \bbG_m) , \qquad \rH^2_{\mathrm{\'et}}(\mathfrak X; \bbG_m), \qquad \lim_n \rH^2_{\mathrm{\'et}}(X_n; \bbG_m) . \]
The first and the last one need no further explanation.
The middle one is defined as:
\[ \rH^2_{\mathrm{\'et}}(\mathfrak X; \bbG_m) \coloneqq \pi_0( \dAz_0(\mathfrak X) ) \coloneqq \pi_0 ( \Map(\mathfrak X, \rK(\bbG_m,2))) . \]

\begin{rem}
	It is possible to reformulate $\rH^2_{\mathrm{\'et}}(\mathfrak X; \bbG_m)$ in more classical terms as follows.
	First, we remark that for every scheme $Y$ there is a canonical quasi-isomorphism (in \emph{cohomological} notation):
	\begin{equation}\label{eq:cohomology_with_coefficients_in_Eilenberg_Maclane}
		\Map(Y, \rK(\bbG_m,2)) \simeq \tau^{\le 2} \rR \Gamma\et(Y; \bbG_m)[2] .
	\end{equation}
	This quasi-isomorphism arises as follows: let $\Sh_{\DAb}(Y_{\mathrm{\'et}}, \tauet)$ (resp.\ $\Sh_{\mathcal D^{\le 0}(\Ab)}(Y_{\mathrm{\'et}};\tauet)$) be the $\infty$-category of \'etale sheaves on $Y$ with values in the unbounded (resp.\ connective) derived $\infty$-category of abelian groups.
	Let $\pi \colon (Y\et,\tauet) \to *$ be the canonical functor.
	We obtain the following commutative diagram:
	\[ \begin{tikzcd}
		\Sh_{\cD^{\le 0}(\Ab)}(Y_{\mathrm{\'et}};\tauet) & \Sh_{\DAb}(Y_{\mathrm{\'et}}, \tauet) \arrow{l}[swap]{\tau^{\le 0}} \\
		\cD^{\le 0}(\Ab) \arrow{u}{{}^{\le 0} \pi^*} & \DAb . \arrow{l}[swap]{\tau^{\le 0}} \arrow{u}{\pi^*}
	\end{tikzcd} \]
	One easily verifies that:
	\[ \Map(Y, \rK(\bbG_m,2)) \simeq {}^{\le 0} \pi_*(\bbG_m[2]) . \]
	At this point, the quasi-isomorphism \eqref{eq:cohomology_with_coefficients_in_Eilenberg_Maclane} is simply induced by the Beck-Chevalley transformation associated to the above diagram.
	Having this identification at one's disposal, it is easy to provide a canonical identification
	\[ \rH^2_{\mathrm{\'et}}(\mathfrak X; \bbG_m) \simeq \rH^2\left( \lim_n \tau_{\le 2} \rR \Gamma\et(X_n; \bbG_m) \right) . \]
\end{rem}

\vspace{0.2cm}

The relationship between $\dBr(\fX)$ and $\displaystyle \lim_n \dBr(X_n)$ is easy to understand in general:

\begin{prop} \label{prop:Milnor_sequence}
	There are short exact sequences
	\[ 0 \longrightarrow {\lim_n}^1 \: \Pic(X_n) \longrightarrow \dBr(\mathfrak X) \longrightarrow \lim_n \dBr(X_n) \longrightarrow 0 \]
	and
	\[ 0 \longrightarrow {\lim_n}^1 \: \Pic(X_n) \longrightarrow \rH_{\mathrm{\'et}}^2( \mathfrak X; \bbG_m ) \longrightarrow \lim_n \rH^2_{\mathrm{\'et}}(X_n; \bbG_m) \longrightarrow 0 . \]
	On the other hand, the natural map
	\[ \rH^1_{\mathrm{\'et}}(\mathfrak X; \Z) \longrightarrow \lim_n \rH^1_{\mathrm{\'et}}(X_n; \Z) \]
	is an isomorphism.
\end{prop}

\begin{proof}
	Applying Milnor's short exact sequence to the inverse limit
	\[ \dAz(\mathfrak X) \simeq \lim_n \dAz(X_n) , \]
	we obtain
	\[ 0 \longrightarrow {\lim_n}^1 \: \pi_1( \dAz(X_n) ) \longrightarrow \dBr(\mathfrak X) \longrightarrow \lim_n \dBr(X_n) \longrightarrow 0 . \]
	Using To\"en's theorem \cref{thm:Toen_dAz}, we see that
	\[ \pi_1( \dAz(X_n) ) \simeq \Pic(X_n) \times \rH^0_{\mathrm{\'et}}(X_n;\mathbb Z) . \]
	Observe that
	\[ {\lim_n}^1 \: ( \Pic(X_n) \times \rH^0_{\mathrm{\'et}}(X_n;\mathbb Z) ) \simeq {\lim_n}^1 \: \Pic(X_n) \times {\lim_n}^1 \: \rH^0_{\mathrm{\'et}}(X_n;\mathbb Z) , \]
	and since the transition maps $\rH^0_{\mathrm{\'et}}(X_{n+1};\mathbb Z) \to \rH^0_{\mathrm{\'et}}(X_n;\mathbb Z) )$ are isomorphisms, we deduce that
	\[ {\lim_n}^1 \: \pi_1(\dAz(X_n)) \simeq {\lim_n}^1 \Pic(X_n) . \]
	The conclusion follows.
	The other statements result from the same reasoning.
\end{proof}

In \cite[Theorem 7.2]{Geisser_Morin_Kernel_Brauer_Manin} it is shown that the canonical map
\[ \rH^2\et(X;\bbG_m) \longrightarrow \rH^2\et(\fX;\bbG_m) \]
is injective when $X$ is regular, $S = \Spec(\Z_p)$ and the map $X \to S$ is flat.
Using derived Azumaya algebras, we can remove most of these assumptions, as we are going to see.

\subsection{Full faithfulness of categorical formal GAGA}

We keep the same notations and assumptions introduced at the beginning of this section.
We have:

\begin{thm} \label{thm:GAGA_Morita}
	The symmetric monoidal functor
	\[ v_\bullet^* \colon \NcSmPr(X) \longrightarrow \lim_{n \in \bbN} \NcSmPr(X_n) . \]
	is fully faithful.
	In other words, if $\cC$ is a smooth and proper category over $X$ and $\cC_n \coloneqq \cC \otimes_{\QCoh(X)} \QCoh(X_n)$, the natural map
	\[ \cC \longrightarrow \lim_n \cC_n \]
	is an equivalence, where the limit is computed in $\PrLomega_X$.
\end{thm}

\begin{proof}
	By assumption, $X$ is categorically proper, i.e.\ the canonical functor
	\[ \Perf(X) \longrightarrow \lim_{n \in \bbN} \Perf(X_n) \]
	is a symmetric monoidal equivalence in $\Cat_\infty$.
	Applying $\Ind$, we deduce that
	\[ \QCoh(X) \longrightarrow \lim_{n \in \bbN} \QCoh(X_n) \]
	is an equivalence in $\PrLomega$, and hence in $\PrLomega_R$.\footnote{Beware that the limit has to be computed in $\PrLomega$ and not in $\PrL$ for this statement to be true.}
	Since $\cC$ is dualizable in $\PrLomega_X$, the functor $\cC \otimes_{\QCoh(X)} -$ commutes with arbitrary limits, so that the canonical map
	\[ \cC \longrightarrow \lim_n \cC \otimes_{\QCoh(X)} \QCoh(X_n) \]
	is an equivalence in $\PrLomega_X$.
\end{proof}

\begin{rem}
	\hfill
	\begin{enumerate}\itemsep=0.2cm
		\item Statement (1) of \cref{thm:GAGA_Morita} is a global counterpart of \cite[Theorem 11.4.4.1]{Lurie_SAG}.
		Notice that if one further assumes that $X$ is smooth and proper over $\Spec(R)$, it would be possible to deduce our result from the cited theorem, using \cite[Theorem 11.3.6.1]{Lurie_SAG} as a stepping stone.
		
		\item We do not expect the map $v_\bullet^* \colon \NcSmPr(X) \to \lim_n \NcSmPr(X_n)$ to be essentially surjective, even when $X = \Spec(R)$.
		Loosely speaking, the reason is that the stack of smooth and proper stable $\infty$-categories is not expected to be a geometric stack, and more precisely integrability should fail. See \cite[\S 8.1]{Anel_Toen}.
	\end{enumerate}
\end{rem}

\begin{cor} \label{cor:injectivity_derived_Brauer_complete_base}
	\hfill
	\begin{enumerate}\itemsep=0.2cm
		\item The natural map
		\[ \dAz^{\mathrm{cat}}(X) \longrightarrow \dAz^{\mathrm{cat}}(\mathfrak X) \]
		is fully faithful.
		
		\item \emph{(Formal injectivity)} The natural map
		\[ \dBr(X) \longrightarrow \dBr(\mathfrak X) \]
		is injective.
		
		\item If ${\displaystyle \lim_n}^1 \Pic(X_n) = 0$, then the natural maps
		\[ \dBr( X ) \longrightarrow \lim_n \dBr( X_n ) \quad \textrm{and} \quad \rH^2_{\mathrm{\'et}}( X ; \bbG_m ) \longrightarrow \lim_n \rH^2_{\mathrm{\'et}}(X_n; \bbG_m) \]
		are injective.
		
		\item In the above proof the assumption that $X$ is a proper derived scheme over $X$ is only 
	\end{enumerate}
\end{cor}

\begin{proof}
	Recall from \cref{thm:GAGA_Morita} that the natural map
	\[ \NcSmPr(X) \longrightarrow \lim_{n \in \mathbb N} \NcSmPr(X_n) \]
	is fully faithful and symmetric monoidal.
	Passing to invertible objects and applying To\"en's \cref{thm:Toen_invertible_objects}, we deduce that the comparison map of point (1) is fully faithful as well.
	Taking $\pi_0$, we immediately obtain statement (2).
	Finally, we observe that (3) follows directly from point (2) and \cref{prop:Milnor_sequence} as the map in question factors naturally as
	\[ \dBr(X) \longrightarrow \dBr(\mathfrak X) \longrightarrow \lim_n \dBr(X_n) . \]
\end{proof}

\begin{rem}\label{rem:failure_surjectivity}
	A natural question to ask is whether the natural comparison map
	\[ \dBr(X) \longrightarrow \dBr(\fX) \]
	is surjective in addition to being injective.
	In general, this will not be true.
	First of all, remark that if the above map is an isomorphism, then separately
	\[ \rH^2\et(X;\bbG_m) \longrightarrow \rH^2\et(\fX; \bbG_m) \quad \textrm{and} \quad \rH^1\et(X;\Z) \longrightarrow \rH^1\et(\fX;\Z) \]
	would be isomorphisms.
	However, \cref{prop:Milnor_sequence} shows that
	\[ \rH^1\et(\fX;\Z) \simeq \lim_n \rH^1\et(X_n;\Z) , \]
	and the topological invariance of the étale site implies that the limit is actually constant and isomorphic to $\rH^1\et(X_0;\Z)$.
	Assume that $X$ is regular, but that the special fiber $X_0$ is not.
	\personal{Probably normal is already enough.}
	Then $\rH^1\et(X;\Z) = 0$ (cf.\ \cite{Deninger_Proper_base_change}), while generically $\rH^1\et(X_0;\Z)$ will not vanish. For example, $X$ can be a regular scheme, flat and of relative dimension $1$ over $S$ such that the special fiber $X_0$ is a nodal curve. 
\end{rem}

In the local case, the injectivity statement of \cref{cor:injectivity_derived_Brauer_complete_base}-(3) holds true under less stringent restrictions on $S$ than completeness:

\begin{cor} \label{cor:injectivity_derived_Brauer_henselian_base}\todo{Modify the statement so that formal injectivity holds over Henselian base without lim1 restriction.}
	Let $S$ be the spectrum of a local noetherian henselian ring $(R, \mathfrak m)$ such that the natural map
	\[ R \longrightarrow \lim_n R / \mathfrak m^n \]
	has regular geometric fibers (This is in particular satisfied when $R$ is quasi-excellent).
	Let $p \colon X \to S$ be a proper derived scheme.
	If ${\displaystyle \lim_n}^1 \Pic(X_n) = 0$, then the natural maps
	\[ \dBr( X ) \longrightarrow \lim_n \dBr( X_n ) \quad \textrm{and} \quad \rH^2_{\mathrm{\'et}}( X ; \bbG_m ) \longrightarrow \lim_n \rH^2_{\mathrm{\'et}}(X_n; \bbG_m) \]
	are injective.
\end{cor}

\begin{proof}
	Thanks to \cref{prop:Milnor_sequence} (which holds without the completeness assumption on $R$), it is enough to prove the statement concerning $\rH^2_{\mathrm{\'et}}(X; \bbG_m)$.
	Consider the commutative triangle
	\[ \begin{tikzcd}[column sep = small]
		\rH^2_{\mathrm{\'et}}(X; \bbG_m) \arrow{rr} \arrow{dr} & & \rH^2\et(\fX; \bbG_m) \arrow{dl} \\
		{} & \lim_n \rH^2\et(X_n, \bbG_m) .
	\end{tikzcd} \]
	We now observe that \cite[Lemma 2.1.3]{Bouthier_Cesnavicius_Torsors_Loop} (applied to the functor sending an $R$-algebra $A$ to $\rH^2_{\mathrm{\'et}}(X \times_S \Spec(A); \bbG_m)$) implies that the horizontal map is injective.
	Since ${\displaystyle \lim_n}^1 \Pic(X_n) = 0$, \cref{cor:injectivity_derived_Brauer_complete_base}-(3) implies that the diagonal arrow on the right is injective.
	The conclusion therefore follows.
\end{proof}

\begin{rem} \label{rem:improving_Grothendieck}
	The problem of the injectivity of the natural map
	\[ \rH^2_{\mathrm{\'et}}(X; \bbG_m) \longrightarrow \lim_n \rH^2_{\mathrm{\'et}}(X_n; \bbG_m) \]
	has been considered as early as in \cite[Lemma III.3.3]{Grothendieck_Dix_expose}.
	There, Grothendieck assumes that
	\[ {\lim_n}^1 \: \Pic(X_n) = 0 \]
	and proves the injectivity under a number of extra restrictions:
	\begin{enumerate}\itemsep=0.2cm
		\item first of all, he has to assume the base to be an henselian, quasi-excellent DVR.
		This relies on the use of smoothing theorems that at the time \cite{Grothendieck_Dix_expose} was written were available only in dimension $1$.
		It is ultimately Popescu's smoothing theorem that allows to remove the restriction on the dimension (see the proof of \cite[Lemma 2.1.3]{Bouthier_Cesnavicius_Torsors_Loop}).
		
		\item In second place, $X$ is assumed to be regular.
		The regularity assumption is used in loc.\ cit.\ to guarantee that $\mathrm{Br}_{\mathrm{Az}}(X) = \rH^2_{\mathrm{\'et}}(X; \bbG_m)$.
		Having this identification, Grothendieck proceeds to prove the injectivity of the map
		\[ \mathrm{Br}_{\mathrm{Az}}(X) \longrightarrow \lim_n \mathrm{Br}_{\mathrm{Az}}(X_n) \hookrightarrow \lim_n \rH^2_{\mathrm{\'et}}(X; \bbG_m) \]
		without further using the regularity assumption.
		The key to his argument is the possibility of representing every class in $\mathrm{Br}_{\mathrm{Az}}(X)$ as the Morita equivalence class associated to a (classical) Azumaya algebra over $X$.
		Philosophically, we can bypass this issue thanks to the use of derived Azumaya algebras and the fact that every class in $\rH^2_{\mathrm{\'et}}(X; \bbG_m)$ can be represented by such an object.
		In practice, this is achieved via our \cref{thm:GAGA_Morita}.
		A different approach is also possible, interpreting a class in $\rH^2_{\mathrm{\'et}}(X; \bbG_m)$ as $\bbG_m$-gerbes, and proving a GAGA theorem for twisted sheaves.
		This will be the content of \cref{subsec:GAGA_G-gerbes}.
		
		\item Finally, Grothendieck also has to assume that the map $p \colon X \to S$ is flat.
		This is used as a technical assumption in the middle of the proof.
		In our context, dropping the flatness assumption has the effect that the schemes $X_n$ become \emph{derived}.
		In our framework this is at best a minor inconvenience, but this language was of course not available at the time \cite{Grothendieck_Dix_expose} was written.
		It is worth observing that the possibility of removing this flatness assumption had been contemplated in \cite[Remark III.3.4-(a)]{Grothendieck_Dix_expose}.
	\end{enumerate}
\end{rem}

\begin{rem}
	Continuing point (2) of the previous remark, the use of derived Azumaya algebras (or, alternatively, of $\bbG_m$-gerbes) allows to prove the injectivity result also for \emph{non-torsion} classes (always under the assumption that $\lim^1 \Pic(X_n) = 0$).
	As we already remarked, the proof of Grothendieck relied explicitly on the possibility of representing every class in $\mathrm{Br}_{\mathrm{Az}}(X)$ via a classical Azumaya algebra over $X$.
	As a result, Grothendieck's approach can \emph{at best} yield the injectivity result for torsion classes in $\mathrm{Br}(X)$.
\end{rem}

\begin{cor}\label{cor:injectivity_dimension1}
	Under the same assumptions of \cref{cor:injectivity_derived_Brauer_henselian_base}, if the relative dimension of $p \colon X \to S$ is at most $1$, then the natural maps
	\[ \dBr(X) \longrightarrow \lim_n \dBr(X_n), \qquad \rH^2\et(X; \bbG_m) \longrightarrow \lim_n \rH^2\et(X_n; \bbG_m) \]
	are injective.
\end{cor}

\begin{proof}
	The obstruction to lift an element in $\Pic(X_{n-1})$ to $\Pic(X_n)$ lies in $\rH^2(X_{n-1}, q^*( \mathfrak m^n / \mathfrak m^{n-1} ))$, where $q \colon X_n \to \Spec(R / \mathfrak m)$ is the natural map.
	Since each $X_n$ has dimension at most $1$, this group vanish and therefore the map
	\[ \Pic(X_n) \to \Pic(X_{n-1}) \]
	is sujrective.
	This implies that ${\displaystyle \lim_n}^1 \Pic(X_n) = 0$, and therefore the conclusion follows from \cref{cor:injectivity_derived_Brauer_henselian_base}.
\end{proof}

\begin{rem}
	In \cite[Remark III.3.4-(a)]{Grothendieck_Dix_expose} Grothendieck observes that it is unlikely that the map
	\[ \rH^2\et(X;\bbG_m) \longrightarrow \lim_n \rH^2\et(X_n; \bbG_m) \]
	is injective in general.
	In loc.\ cit.\ he further proposes a method to obtain an explicit counterexample.
	The outline of his strategy is the following: starting with Mumford's normal surface $X$, we make it projective and we fiber it over a curve $C$.
	The curve necessarily contains a point $t$ whose fiber $X_t$ supports the non-torsion Brauer class.
	If the non-torsion class survived to the base change to the henselianization $C^h_t$, then indeed one would get a contradiction.
	Nevertheless, the previous corollary shows that base-changing to $C^h_t$, this non-torsion class must become torsion.
	This is not entirely surprising: if instead of base-changing we henselianized the local ring of $X$ at the singular point along any direction over $C$, \cite[Theorem 2.1.7-(b)]{Bouthier_Cesnavicius_Torsors_Loop} would already show that the non-torsion class would become torsion for dimensions reasons (as the cohomological Brauer group of a noetherian $1$-dimensional ring is always torsion).
	
	It is worth observing that our findings (and in particular the short exact sequences of \cref{prop:Milnor_sequence} and the injectivity result \cref{cor:injectivity_derived_Brauer_complete_base}-(2)) are well in line with the general philosophy promoted by Grothendieck.
	Indeed, as soon as the relative dimension is higher than $1$, the obstruction to injectivity is exactly represented by ${\displaystyle \lim_n}^1 \Pic(X_n)$, which generically will not vanish.
\end{rem}

\subsection{A remark on a conjecture of Grothendieck}
Assume once again that $S = \Spec(R)$ is a noetherian, complete local ring and that $p\colon X\to S$ is a proper and flat morphism. 
As we already recalled in Remark \ref{rem:improving_Grothendieck}, in \cite[III, \S 3]{Grothendieck_Dix_expose}, Grothendieck considers the natural map
\begin{equation}\label{eq:Grothendieck_question}
 \mathrm{Br}_{\mathrm{Az}}(X) \longrightarrow \lim_{n\geq 1} \mathrm{Br}_{\mathrm{Az}}(X_n)
\end{equation}
and raises the problem of its injectivity. 
Besides proving it under a certain number of assumptions, he suggests in \cite[Remark III.3.4]{Grothendieck_Dix_expose} that injectivity might hold in general, and in particular without flatness and without the vanishing of ${\lim}^1 \: \Pic(X_n)$. 

Let us rephrase Grothendieck's question as follows.
We continue with the notations of the previous paragraph.
Let $\ell$ be a prime number different from the residue characteristic of $S$, and let 
\[ 
 \rH^2_{\mathrm{\'et}}(\mathfrak X; \mu_\ell)  \coloneqq \pi_0 ( \Map(\mathfrak X, \rK(\mu_\ell,2))).
 \]
Note that since $\ell$ is coprime to the residue characteristic of $S$, the stack $ \rK(\mu_\ell,2)$ satisfies nil-invariance \personal{We might want to explain that $ \rK(\mu_\ell,2)$ classifies derived Azumaya algebras together with a choice of a trivialization of its $\ell$-power. It is not strictly speaking necessary}, so that the transition maps $\rH^2_{\mathrm{\'et}}( X_{n+1}; \mu_\ell) \to \rH^2_{\mathrm{\'et}}( X_{n}; \mu_\ell)$ are all isomorphism. In particular, we have that the canonical map
\[ \rH^2_{\mathrm{\'et}}(\mathfrak X; \mu_\ell) \longrightarrow \lim_n \rH^2_{\mathrm{\'et}}( X_n; \mu_\ell) \]
is an isomorphism.

We relate this group with the derived Brauer group in the following way:

\begin{lem}\label{lem:diagram_Milnor_torsion}There is a commutative diagram
	\begin{equation}\label{eq:diagram_Milnor_torsion}
		\begin{tikzcd} 
			0 \arrow[r] & {\lim_n}^1 \: \Pic(X_n) \arrow[r] &  \rH^2_{\mathrm{\'et}}(\mathfrak X; \bbG_m) \arrow[r] & \lim_n \rH^2_{\mathrm{\'et}}( X_n; \bbG_m) \arrow[r]   & 0\\
			& & \rH^2_{\mathrm{\'et}}(\mathfrak X; \mu_\ell) \arrow[r, "\simeq"] \arrow[u] & \lim_n \rH^2_{\mathrm{\'et}}( X_n; \mu_\ell) \arrow[u]&     \\
			& &  \Pic(X)/\ell  \arrow[r] \arrow[u] & \lim_n (\Pic(X_n)/\ell) \arrow[u]&     \\
			& &  0   \arrow[u] & 0. \arrow[u]&  
		\end{tikzcd}
	\end{equation}
	where the first row is exact, the middle horizontal map is an isomorphism and the columns are exact.
\end{lem}

\begin{proof}
	The first row is simply the Milnor sequence already considered in Proposition \ref{prop:Milnor_sequence}. 
	The central column is obtained from the long exact sequence of homotopy groups associated to the Kummer  sequence $\rK(\mu_\ell,2) \to \rK(\bbG_m,2) \to \rK(\bbG_m,2)$ induced by the $\ell$-power map on $\bbG_m$
	\[	\begin{tikzcd}
		\Pic(\mathfrak X) \arrow[r, "\ell"] & \Pic(\mathfrak X) \arrow[r]  &\rH^2_{\mathrm{\'et}}(\mathfrak X; \mu_\ell) \arrow[r] &  \rH^2_{\mathrm{\'et}}(\mathfrak X; \bbG_m),
	\end{tikzcd} \]
	noting that the group $\Pic(\mathfrak X) $ is isomorphic to $\Pic(X)$ by GAGA. Finally, the right column is obtained by applying the inverse limit functor to the exact sequence 
	\[ \begin{tikzcd}
		0\arrow[r] & \Pic( X_n)/\ell \arrow[r]  &\rH^2_{\mathrm{\'et}}(X_n; \mu_\ell) \arrow[r] &  \rH^2_{\mathrm{\'et}}( X_n; \bbG_m),
	\end{tikzcd} \]
	which holds for every $n\geq 1$.
\end{proof}

Applying the Snake Lemma to the first two (exact) rows of \eqref{eq:diagram_Milnor_torsion}, we obtain in particular an exact sequence
\begin{equation}\label{eq:rho}
	\begin{tikzcd}
		0\arrow[r] &  \Pic(X)/\ell  \arrow[r]  & \lim_n (\Pic(X_n)/\ell) \arrow[r, "\rho"] &  {\lim}^1_{n} \: \Pic(X_n).
	\end{tikzcd}
\end{equation}
The following Lemma is a simple diagram chase.

\begin{lem} \label{lem:rho}
	The kernel of the composite morphism
	\[ \begin{tikzcd}
		\mathrm{Br}(X)[\ell] \arrow[r, hook] &  \rH^2_{\mathrm{\'et}}(\mathfrak X; \bbG_m)[\ell] \arrow[r] &  \lim_n \rH^2_{\mathrm{\'et}}( X_n; \bbG_m)
	\end{tikzcd} \]
	is isomorphic to the image of $\rho$.
\end{lem}

Recall now that by the classical Skolem-Noether Theorem \cite[Theorem 2.5]{Milne_Etale_cohomology}, there is a canonical injective homomorphism  $\mathrm{Br}_{\mathrm{Az}}(X)  \hookrightarrow \mathrm{Br}(X) = \rH^2_{\mathrm{\'et}}( X; \bbG_m)$, whose image is contained in the subgroup of torsion elements of $ \rH^2_{\mathrm{\'et}}( X; \bbG_m)$. 
In fact, if $X$ is a scheme endowed with an ample invertible sheaf, a result of Gabber \cite{de_Jong_Gabber} affirms that the image consists precisely of the torsion elements, i.e. that every torsion class in $\rH^2_{\mathrm{\'et}}( X; \bbG_m)$ can be realized as a (classical) Azumaya algebra on $X$.

\smallskip

For such $X$, a positive answer to Grothendieck's question on the injectivity of \eqref{eq:Grothendieck_question} would imply in particular that for every $\ell$ coprime to the residue characteristic of $S$, the map 
\[ \Pic(X)/\ell \to \lim_n (\Pic(X_n)/\ell) \]
is an isomorphism.

\smallskip

This is not true for an arbitrary proper morphism $p\colon X\to S$.
 Notice that this is clearly the case if the transition maps $\Pic(X_n) \to \Pic(X_{n-1})$ are surjective, for example when the relative Picard functor $\mathbf{Pic}_{X/S}$ is representable by a \emph{smooth} algebraic space over $S$. In this case, however, the whole term ${\displaystyle \lim_n}^1 \Pic(X_n) = 0$ vanishes, and this already implies  a  stronger injectivity result  as discussed in Corollary \ref{cor:injectivity_derived_Brauer_complete_base}-(2).
 However notice that the criterion provided by \cref{lem:rho} is strictly finer than the vanishing of the $\lim^1_n \Pic(X_n)$:
 
 \begin{eg}
 	Let $k$ be a field and let $X \to \Spec(k)$ be a geometrically connected, smooth, projective $k$-scheme with a $k$-rational point.
 	Let $S = \Spec(A)$ be the spectrum of a complete noetherian local $k$-algebra and consider the constant family $X_S$.
 	Let $\mathfrak m$ be its maximal ideal and let $S_n \coloneqq \Spec(A / \mathfrak m^n)$.
 	Then for every prime $\ell$ the canonical map
 	\[ \Pic(X_S) / \ell \longrightarrow \lim_n \Pic(X_{S_n})/\ell \]
 	is an isomorphism, as it was kindly explained to us by A.\ Bouthier.
 	In particular, the map $\rho$ of \eqref{eq:rho} is zero and hence \cref{lem:rho} implies that the map $\mathrm{Br}(X_S) \to \lim_n \mathrm{Br}(X_{S_n})$ is injective.
 	However, there are known examples where the Picard scheme $\underline{\Pic}_{X_S/S}$ is not smooth, and in such cases we are not aware of any method to prove the vanishing of the whole $\lim_n^1 \Pic(X_n)$.
 \end{eg}
 \begin{eg}
     One can further analyse the obstruction given by \cref{lem:rho}, and observe that the presence of torsion classes in ${\lim}^1_{n} \: \Pic(X_n)$ of order coprime to the residue characteristic of $S$ is enough to construct a counterexample to Grothendieck's conjecture. This can be done starting from Shioda's quartic surface \cite{shioda1}, \cite{shioda2} as we learned from A.\ Kresch. 
     We remark that after a first version of this paper was written, the problem of constructing explicit obstructions to Grothendieck's conjecture has been considered extensively in \cite{kresh-mathur}. 
 \end{eg}

\section{Grothendieck existence theorem for $\bbG_m$-gerbes} \label{sec:G-gerbes}

\todo{Rewrite the intro}

Assume again that $S = \Spec(R)$ is a noetherian ring, complete along an ideal $I$.
Let $p \colon X \to S$ is a proper morphism.
Write $S_n \coloneqq \Spec(R / I^n)$ and set $X_n \coloneqq S_n \times_S X$.
As in the previous section we let
\[ \fX \coloneqq \colim_n X_n \simeq \Spf(R,I) \times_S X \]
be the formal completion of $X$ along the special fiber.
In \cref{cor:injectivity_derived_Brauer_complete_base}-(2) we proved that the natural map
\[ \rH^2\et(X;\bbG_m) \longrightarrow \rH^2\et(\fX;\bbG_m) \]
is injective.
This was obtained by interpreting classes in $\rH^2\et(X;\bbG_m)$ as derived Azumaya algebras and, ultimately, as smooth and proper $\infty$-categories linear over $X$.
On the other hand, every class in $\rH^2\et(X;\bbG_m)$ can also be interpreted as a $\bbG_m$-gerbe.
A $\bbG_m$-gerbe on the formal scheme $\fX$ is exactly the given of a sequence $(\mathfrak A_n, \phi_n)_{n \ge 1}$, where each $\mathfrak A_n$ is a $\bbG_m$-gerbe on $X_n$ and $\phi_n$ is an equivalence
\[ \phi_n \colon \mathfrak A_n \simeq X_n \times_{X_{n+1}} \mathfrak A_{n+1} . \]
The formal injectivity therefore can be phrased as follows: if $\mathfrak A$ is a $\bbG_m$-gerbe on $X$ and we are given trivializations $\sigma_n$
\[ \sigma_n \colon \mathfrak A_n \xrightarrow{\sim} \rB \bbG_m \times X_n \]
as $\bbG_m$-gerbes over $X_n$, \emph{together with} homotopies $h_n$ making the diagrams
\begin{equation} \label{eq:trivializations_gerbe}
	\begin{tikzcd}[column sep = large]
		\mathfrak A_n \arrow{d}{\sigma_n} \arrow{r}{\phi_n} & \mathfrak A_{n+1} \arrow{d}{\sigma_{n+1}} \\
		\rB \bbG_m \times X_n \arrow{r}{\id_{\bbG_m} \times i_n} & \rB \bbG_m \times X_{n+1} ,
	\end{tikzcd}
\end{equation}
then there exists a global trivialization $\sigma \colon \mathfrak A \simeq \rB \bbG_m \times X$ as $\bbG_m$-gerbes over $X$.
The trivializations $\sigma_n$ encode line bundles $\cL_n \in \Pic(\mathfrak A_n)$, and the homotopies $h_n$ allow to promote these line bundles to a formal system
\[ (\cL_n)_{n \geqslant 0} \in \lim_{n \geqslant 0} \Pic(\mathfrak A_n) \ . \]
If we had at our disposal Grothendieck's existence theorem for $\cA$, this would allow to construct a global line bundle on $\mathfrak A$, which would be easily checked to be a trivialization.
However, Grothendieck's existence theorem is only known for $\bbG_m$-gerbes that have global resolution property (see \cite[Corollary 1.7]{Alper_Hall_Rydh_Etale}).
On the other hand, as proven by Totaro in \cite[Theorem 1.1]{Totaro_Resolution_property}, this is possible if and only if the $\bbG_m$-gerbe is a quotient stack.
In turn, it was shown in \cite[Theorem 3.6]{Vistoli_Brauer_quotient_stack} that asking that a $\bbG_m$-gerbe is a quotient stack is equivalent to ask that the associated class $\alpha \in \rH^2\et(X;\bbG_m)$ is representable by a classical Azumaya algebra.

\medskip

The goal of this section is to exploit our \cref{thm:GAGA_Morita} to deduce that Grothendieck's existence theorem holds for arbitrary $\bbG_m$-gerbes, see \cref{cor:GAGA_gerbes} below.
The key ingredient needed to deduce this result from \cref{thm:GAGA_Introduction} is the character decomposition for sheaves on a $\bbG_m$-gerbe, originally obtained by Lieblich \cite{Lieblich_Twisted_period_index} and extended to the level of triangulated categories in \cite{Bergh_Schnurer_Gerbes}.
We briefly review this theory, taking the opportunity to recast the result of \cite{Bergh_Schnurer_Gerbes} in an $\infty$-categorical language, which is more flexible and more adapted to the computations we have to make.

\subsection{Character decomposition for trivial $\bbG_m$-gerbes}

The starting point is the trivial $\bbG_m$-gerbe over a base.
The simplest is to start working relative to the sphere spectrum $\bbS \in \Sp$, but the reader inexperienced with spectral algebraic geometry can safely replace $\bbS$ by $\Z$: every construction will go through.
Let $S \coloneqq \Spec(\mathbb S) \in \dSt_\bbS$ be the associated spectral stack.
Denote by $\rB\bbG_m$ the classifying stack of the \emph{flat} multiplicative group scheme.
Recall from \cite[Theorem 4.1]{Moulinos_Filtrations} that the following theorem holds true:

\begin{thm}\label{thm:graded_spectra}
	There exists a symmetric monoidal equivalence
	\[ \Phi \colon \QCoh(\rB\bbG_m) \simeq \Fun(\Z, \Sp) \ , \]
	where $\Z$ denotes the \emph{set} of integers, and where the tensor structure of the right hand side is given by Day's convolution product.
\end{thm}

It is easy to extend this result to the case where the base is no longer $S$, using the following standard fact:

\begin{lem} \label{prop:box_product_derived_stacks}
	Let $F$ and $G$ be two derived stacks.
	Assume that $\QCoh(G)$ is compactly generated.
	Then the canonical box product
	\[ \boxtimes \colon \QCoh(F) \otimes \QCoh(G) \longrightarrow \QCoh(F \times G) \]
	is an equivalence.
\end{lem}

\begin{proof}
	Observe that the box product is (contravariantly) functorial in both $F$ and $G$.
	Because $\QCoh(G)$ is compactly generated, the functor $- \otimes \QCoh(G)$ commutes with arbitrary limits in $\PrL$.
	We can therefore reduce ourselves to the case where $F$ is affine.
	Say $F = \Spec(A)$.
	Then $\QCoh(F) \simeq \Mod_A$ is again compactly generated, so repeating the same argument reduces to the case where $G$ is also affine, say $G \simeq \Spec(B)$.
	In this case, we have to check that the box product
	\[ \boxtimes \colon \Mod_A \otimes \Mod_B \longrightarrow \Mod_{A\otimes B} \]
	is an equivalence, and this is a particular case of \cite[Theorem 4.8.5.16-(4)]{Lurie_Higher_algebra}.
\end{proof}

\begin{cor}
	Let $X$ be a derived stack.
	Then there are canonical equivalences
	\[ \QCoh(\rB\bbG_m \times X) \simeq \QCoh(\rB\bbG_m) \otimes \QCoh(X) \simeq \Fun(\Z, \QCoh(X)) \ . \]
\end{cor}

Let us fix some notation.

\begin{notation}
	Let $X$ be a derived stack and let $d \in \Z$ be an integer.
	\begin{itemize}\itemsep=0.2cm
		\item We let $\cO_X(d) \in \QCoh(\rB\bbG_m \times X)$ be the object corresponding to $\cO_X$ concentrated in weight $d$ under the equivalence $\QCoh(\rB\bbG_m \times X) \simeq \Fun(\Z, \rB\bbG_m)$.
		
		\item For $\cF \in \QCoh(\rB\bbG_m \times X)$, we write $\cF(d) \coloneqq \cF \otimes \cO_{X}(d)$.
		
		\item Write $\pi_X \colon \rB\bbG_m \times X \to X$ for the canonical projection.
		We let
		\[ \mathrm T_{X,d} \coloneqq \pi_X^\ast(-)(d) \colon \QCoh(X) \longrightarrow \QCoh(\rB\bbG_m \times X) \]
		be the functor sending $\cF$ to $\pi^\ast(\cF)(d)$.
		Similarly, we let
		\[ \mathrm{wt}_{X,d} \coloneqq \pi_\ast((-)(-d)) \colon \QCoh(\rB\bbG_m) \longrightarrow \Sp \]
		be the functor sending a sheaf $\cG \in \QCoh(X \times \rB\bbG_m)$ to $\pi_{X,\ast}(\cG(-d))$.
	\end{itemize}
\end{notation}

\begin{rem}
	Notice that for every integer $d \in \Z$, we have canonical identifications
	\[ \mathrm T_{X,d} \simeq \mathrm T_d \otimes \id_{\QCoh(X)} \qquad \text{and} \qquad \mathrm{wt}_{X,d} \simeq \mathrm{wt}_d \otimes \id_{\QCoh(X)} \ , \]
	where $\mathrm T_d$ and $\mathrm{wt}_d$ denote the same functors for the the base $S = \Spec(\bbS)$.
	In particular, for $\cF \in \QCoh(X)$ we have
	\[ \mathrm T_{X,d}(\cF) \simeq \cF \otimes \mathrm T_{X,d}(\cO_X) \simeq \cF \otimes \cO_X(d) \ .  \]
	Committing a slight abuse of notation, we will often write $\cF(d)$ instead of $\mathrm T_d(\cF)$.
\end{rem}

\begin{lem}\label{lem:graded_spectra_standard_facts}
	Let $d \in \Z$ be an integer and let $X$ be a derived stack.
	Then the functor $\mathrm T_{X,d} \colon \QCoh(X) \to \QCoh(\rB\bbG_m \times X)$ is fully faithful.
	Furthermore, it is both left and right adjoint to the functor $\mathrm{wt}_{X,d}$.
	Finally, both $\mathrm T_{X,d}$ and $\mathrm{wt}_{X,d}$ commute with compact objects.
\end{lem}

\begin{proof}
	Under the equivalence $\QCoh(\rB\bbG_m \times X) \simeq \Fun(\Z, \QCoh(X))$, the functor $\mathrm T_{X,d}$ corresponds to the left Kan extension along $j_d \colon \{d\} \hookrightarrow \Z$.
	Since $\Sp$ is pointed and $\Z$ is discrete, left Kan extension along $j_d$ coincides with right Kan extension.
	On the other hand, $\mathrm{wt}_{X,d}$ is canonically identified with restriction along $j_d$.
	Thus, all statements follow.
\end{proof}

\begin{rem}
	When $d = 0$, the above lemma guarantees that the functor
	\[ \mathrm{wt}_0 = \pi_\ast \colon \QCoh(\rB\bbG_m) \longrightarrow \Sp \]
	commutes with colimits.
	In particular, $\cO_{\rB\bbG_m}$ is compact, and therefore every perfect complex on $\QCoh(\rB\bbG_m)$ is compact as well.
	Notice as well that \cref{thm:graded_spectra} guarantees that $\QCoh(\rB\bbG_m)$ is compactly generated.
	If instead of working over the sphere spectrum we chose to work over an underived base, the same would follow from \cite[Example 8.6]{Rydh_Hall_Perfect_complexes}
\end{rem}

\subsection{Inertial actions and bandings} \label{subsec:inertial_action}

We need a brief digression concerning inertial actions, which are the fundamental ingredient to both define gerbes and to provide a generalization of \cref{thm:graded_spectra} for a general $\bbG_m$-gerbe.

\medskip

Let $f \colon Y \to X$ be a morphism of derived stacks.
We let $\rI_Y^\bullet X \coloneqq \Cech(f)$ be the \v{C}ech nerve of $f$.
Notice that \cite[Proposition 6.1.2.11]{HTT} shows that this is a groupoid object in the sense of \cite[Definition 6.1.2.7]{HTT}.
We refer to it as the \emph{self-intersection groupoid of $Y$ inside $X$}.
We let
\[ \rB_Y X \coloneqq | \rI_Y^\bullet X | \]
be the geometric realization of the self-intersection groupoid.
It coincides with its classifying groupoid.
By construction, there is a canonical map
\begin{equation}\label{eq:general_multiplication_map}
	\mu \colon \rB_Y X \longrightarrow X \ .
\end{equation}
Notice that this map is sensitive to the infinitesimal geometry of $Y$ inside $X$, and in particular does \emph{not} factor as a projection to $Y$ followed by the map to $X$.
Observe that the atlas map $u \colon Y \to \rB_Y X$ is by definition an epimorphism.
Since groupoids are effective in $\dSt$, we deduce that the \v{C}ech nerve $\Cech(u)$ is canonically isomorphic to $\rI_Y^\bullet X$.
Unraveling the definitions, we find more precisely:

\begin{lem}\label{lem:computing_mu}
	Then the map $\mu$ induces a transformation $\mu_\bullet \colon \Cech(u) \to \Cech(f)$ of simplicial objects which is homotopic to the identity.
\end{lem}

Let $Y \to X$ again be a morphism of derived stacks and consider the canonical diagonal map $\delta_{Y/X} \colon Y \to Y\times_X Y$.
We refer to the self-intersection groupoid of $\delta_{Y/X}$ as the \emph{inertia groupoid of $Y$ relative to $X$}.
We denote it by $\rI^\bullet(Y/X)$ and when $X$ is clear out of the context we simply write $\rI^\bullet Y$.
This is an object in $\Fun(\mathbf \Delta\op, \dSt_{/Y \times_X Y})$.
Notice that we have
\[ \rI^0(Y/X) \simeq Y , \qquad \rI^1(Y/X) \simeq Y \times_{Y \times_X Y} Y . \]
In particular, this simplicial object has the usual inertia stack $\rI(Y/X)$ of $Y$ relative to $X$ as stack of morphisms, and $Y$ as stack of objects.
Let $p_1, p_2 \colon Y \times_X Y \to Y$ be the two projections.
Composing with $p_1$ yields a groupoid object in $\Fun(\mathbf \Delta\op, \dSt_{/Y})$, which is easily seen to be a group object.
We can therefore think of $\rI^\bullet(Y/X)$ as a group structure on $\rI(Y/X)$.
Beware however that the choice of the projection is important, as using $p_2$ instead of $p_1$ yields the opposite group structure.
Notice further that the forgetful functor $\dSt_{/Y \times_X Y} \to \dSt_{/Y}$ commutes with colimits.
In particular, computing the geometric realization in $\dSt_{/Y \times_X Y}$ or in $\dSt_{/Y}$ produces the same output, but with the difference that doing it in the second category allows us to identify $\rB_Y( Y \times_X Y )$ with the classifying stack of the group object $\rI(Y/X)$.
For this reason, we simply denote this object by $\rB_Y(\rI(Y/X))$.
Finally, we define the map $\mathrm{act} \colon \rB_Y(\rI(Y/X)) \to X$ as the composition
\[ \mathrm{act} \colon \rB_Y(\rI(Y/X)) \simeq \rB_Y( Y \times_X Y ) \stackrel{\mu}{\longrightarrow} Y \times_X Y \stackrel{p_2}{\longrightarrow} Y . \]
This map is not the projection, but the composite
\[ Y \longrightarrow \rB_Y(\rI(Y/X)) \xrightarrow{\mathrm{act}} Y \]
is canonically equivalent to the identity.
This implies that pulling back along $\mathrm{act}$ we can endow objects (e.g.\ quasi-coherent sheaves) over $Y$ with a canonical right action of the inertia stack $\rI(Y/X)$.\\

We now turn to the notion of banding.
We start with the following easy construction:

\begin{construction}
	Let $q \colon Y \to X$ be a morphism of derived stacks.
	The functor
	\[ q^* \colon \dSt_{/X} \longrightarrow \dSt_{/Y} \]
	given by pullback along $q$ admits a \emph{right} adjoint $q_* \colon \dSt_{/Y} \to \dSt_{/X}$.
	\personal{This is \emph{not} given by composition with $q$!}
	Since $q^*$ is monoidal (with respect to the cartesian structures), $q_*$ is lax-monoidal.
	In particular, it gives rise to an adjunction
	\[ q^* \colon \mathrm{Mon}_{\mathbb E_1}^{\mathrm{gp}}( \dSt_{/X} ) \leftrightarrows \mathrm{Mon}_{\mathbb E_1}^{\mathrm{gp}}( \dSt_{/Y} ) \colon q_* . \]
\end{construction}

\begin{defin} \label{defin:banding}
	Let $q \colon Y \to X$ be a morphism of derived stacks and let $\mathrm G$ be a derived group stack over $X$.
	A \emph{weak $\mathrm G$-banding on $Y$} is a morphism
	\[ \alpha \colon \mathrm G \longrightarrow q_*( \rI(Y/X) ) \]
	in $\mathrm{Mon}_{\mathbb E_1}^{\mathrm{gp}}( \dSt_{/X} )$, where $\rI(Y/X)$ is considered as a group via the structure introduced in \cref{subsec:inertial_action}.
	A weak $\mathrm G$-banding $\alpha$ is said to be a \emph{$\mathrm G$-banding} if $\alpha$ is an equivalence.
\end{defin}

Let $q \colon Y \to X$ be a morphism of derived stacks and let $\alpha$ be a weak $\mathrm G$-banding on $Y$.
By adjunction, it corresponds to a morphism of group stacks over $Y$
\[ q^* \mathrm G \longrightarrow \rI(Y/X) . \]
Applying the deloop functor, we obtain a canonical map
\[ \varpi_\alpha \colon q^*( \rB_X(\mathrm G) ) \simeq \rB_Y( q^* \mathrm G ) \longrightarrow \rB_Y( \rI(Y/X) ) . \]
Composing with $\mathrm{act} \colon \rB_Y(\rI(Y/X)) \to Y$ we obtain a map
\begin{equation}\label{eq:act}
	\mathrm{act}_\alpha \coloneqq \mathrm{act} \circ \varpi_\alpha \colon q^*(\rB_X(\mathrm G)) \longrightarrow Y .
\end{equation}
Observe that the composition of $\mathrm{act}_\alpha$ with the atlas $u \colon Y \to q^*(\rB_X(\mathrm G))$ is canonically equivalent to the identity of $Y$.
In particular, we have:

\begin{lem} \label{lem:act_t-exact}
	Assume that $\mathrm G$ is flat group stack over $X$.
	Then the functor
	\[ \mathrm{act}_\alpha^* \colon \QCoh(Y) \longrightarrow \QCoh( q^*(\rB_X(\mathrm G)) ) \]
	is flat.
\end{lem}

\begin{proof}
	Since $\mathrm G$ is flat over $X$, the induced map $Y \to q^*( \rB_X(\mathrm G) )$ is flat as well.
	Therefore the conclusion simply follows from the fact that $\mathrm{act}_\alpha \circ u$ is the identity of $Y$.
\end{proof}

\begin{eg}\label{eg:act_trivial_gerbe}
	We unravel the construction of $\mathrm{act}_\alpha$ in the special case where $X = \Spec(\mathbb S)$ and $Y = \rB\bbG_m$.
	We have a canonical identification
	\[ \mathrm I(\rB\bbG_m) \simeq \rB\bbG_m \times \bbG_m \ , \]
	where the group structure relative to $\rB\bbG_m$ is simply the one induced by the absolute group structure of $\bbG_m$.
	We have
	\[ q_\ast([\bbG_m / \bbG_m]) \simeq \bbG_m \ , \]
	and the banding $\alpha$ in this case is the identity.
	As a consequence, the map $q^\ast(\bbG_m) \to \rB\bbG_m \times \bbG_m$ is the identity as well.
	Consider now the commutative triangle
	\[ \begin{tikzcd}[column sep = small]
		& \rB\bbG_m \arrow{dl}[swap]{u} \arrow{dr}{\Delta} \\
		\rB_{\rB\bbG_m}(\rB\bbG_m \times \bbG_m) \arrow{rr}{\mu} & & \rB\bbG_m \times \rB\bbG_m \ ,
	\end{tikzcd} \]
	where $u$ is the atlas map and $\Delta$ the diagonal.
	Notice that there is a natural identification $\rB_{\rB\bbG_m}(\rB\bbG_m \times \bbG_m) \simeq \rB\bbG_m \times \rB\bbG_m$.
	Under the equivalence $\QCoh(\rB\bbG_m \times \rB\bbG_m) \simeq \Fun(\Z \times \Z, \Sp)$, let us write for $\cO(p,q)$ for the line bundle on $\rB\bbG_m \times \rB\bbG_m$ concentrated in the bigrade $(p,q)$.
	Then a map $f \colon X \to \rB\bbG_m \times \rB\bbG_m$ is completely determined by the pair of line bundles $(f^\ast(\cO(1,0)), f^\ast(\cO(0,1)))$.
	In these terms, the map $\Delta$ is classified by the pair $(\cO_{\rB\bbG_m}(1), \cO_{\rB\bbG_m}(1))$, and the map $u$ by the pair $(\cO_{\rB\bbG_m}(1), \cO_{\rB\bbG_m})$.
	By construction we have
	\[ \mu^\ast(\cO(1,0)) \simeq \cO(1,0) \ . \]
	On the other hand, if $(p,q) \in \Z \times \Z$ is such that $\mu^\ast(\cO(0,1)) \simeq \cO(p,q)$, then the commutativity of the above diagram forces
	\[ \cO(p) \simeq u^\ast(\cO(p,0)) \otimes u^\ast(\cO(0,q)) \simeq u^\ast(\cO(p,q)) \simeq u^\ast \mu^\ast(\cO(0,1)) \simeq \Delta^\ast(\cO(0,1)) \simeq \cO_{\rB\bbG_m}(1) \ , \]
	and therefore $p = 1$.
	On the other hand, write $\mu_\bullet \colon \Cech(u) \to \Cech(\Delta)$ for the map of simplicial objects induced by $\mu$.
	\Cref{lem:computing_mu} guarantees that the map
	\[ \mu_1 \colon \rB\bbG_m \times \bbG_m \longrightarrow \rB\bbG_m \times \bbG_m \]
	is homotopic to the identity.
	On the other hand, unraveling the definitions we see that if $\mu^\ast(\cO(0,1)) \simeq \cO(1,q)$, then $\mu_1$ is induced by the homomorphism of groups $(-)^q \colon \bbG_m \to \bbG_m$.
	Thus, $q = 1$ as well.
	
	\medskip
	
	In conclusion, $\mu^\ast(\cO(0,1)) \simeq \cO(1,1)$.
	It follows that $\mu$ actually coincides with the multiplication map $m \colon \rB\bbG_m \times \rB\bbG_m \to \rB\bbG_m$.
\end{eg}

\subsection{GAGA theorem for twisted sheaves} \label{subsec:GAGA_G-gerbes}

We start with a brief review of the notion of $\mathbb G_m$-gerbes and twisted sheaves (see \cite{Giraud_Cohomologie_1971,Lieblich_Twisted_period_index,Bergh_Schnurer_Gerbes} for more thorough introductions to this language).

\begin{defin}\label{def:gerbe}
	Let $X$ be a derived stack.
	A $\bbG_m$-gerbe on $X$ is a pair $(\mathfrak A, \alpha)$, where:
	\begin{enumerate}\itemsep=0.2cm
		\item $\pi \colon \mathfrak A \to X$ is a derived stack over $X$ such that both the structural map $\pi$ and the diagonal $\delta_\pi \colon \mathfrak A \to \mathfrak A \times_X \mathfrak A$ are epimorphism;
		\item a $\bbG_m$-banding $\alpha \colon X \times \bbG_m \to \pi_* \rI(\mathfrak A / X)$ (see \cref{defin:banding}) .
	\end{enumerate}
\end{defin}

\begin{construction}\label{construction:d_homogeneous_components}
	Let $(\mathfrak A, \alpha)$ be a $\bbG_m$-gerbe over $X$ and fix an integer $d \in \Z$.
	Let
	\[ p \colon \rB\bbG_m \times \mathfrak A \longrightarrow \mathfrak A \qquad \text{and} \qquad u \colon \mathfrak A \longrightarrow \rB\bbG_m \times \mathfrak A \]
	for the canonical projection and the atlas map, respectively.
	We define the functor
	\[ \delta_d \colon \QCoh(\mathfrak A) \longrightarrow \QCoh(\mathfrak A) \]
	by setting
	\[ \delta_d(\cF) \coloneqq u^\ast \big(p^\ast \big(p_\ast(\mathrm{act}_\alpha^\ast(\cF)(-d)\big) (d) \big) \ , \]
	where $\mathrm{act}_\alpha$ is the map \eqref{eq:act}.
	Since $\mathrm{act}_\alpha \circ u \simeq \id_{\mathfrak A}$, the unit of the adjunction $p^\ast \dashv p_\ast$ induces a canonical natural transformation
	\[ j_d \colon \delta_d \longrightarrow \id_{\QCoh(\mathfrak A)} \ . \]
\end{construction}

\begin{defin}
	Let $X$ be a derived stack and let $(\mathfrak A, \alpha)$ be a $\bbG_m$-gerbe on $X$.
	Let $d \in \Z$ be an integer.
	We say that a quasi-coherent sheaf $\cF \in \QCoh(\mathfrak A)$ is \emph{$d$-homogeneous} if the canonical map $j_{d,\cF} \colon \delta_d(\cF) \to \cF$ is an equivalence.
	We let $\QCoh_d(\mathfrak A)$ denote the full subcategory of $\QCoh(\mathfrak A)$ spanned by $d$-homogeneous sheaves.
\end{defin}

\begin{notation}
	For a $\bbG_m$-gerbe $(\mathfrak A, \alpha)$ on a derived stack $X$ and an integer $d \in \Z$, we set $\Perf_d(\mathfrak A) \coloneqq \Perf(\mathfrak A) \cap \QCoh_d(\mathfrak A)$.
\end{notation}

\begin{lem}\label{lem:d_homogeneous_basic_facts}
	Let $(\mathfrak A, \alpha)$ be a $\bbG_m$-gerbe over a derived stack $X$.
	Fix an integer $d \in \Z$.
	Then:
	\begin{enumerate}\itemsep=0.2cm
		\item The functor $\delta_d$ is $t$-exact.
		In particular a quasi-coherent sheaf $\cF \in \QCoh(\mathfrak A)$ is $d$-homogeneous if and only if $\pi_i(\cF)$ is $d$-homogeneous for every $i \in \Z$.
		
		\item If
		\[ X \simeq \colim_{i \in I} X_i \]
		inside $\dSt$, then
		\[ \QCoh_d(\mathfrak A) \simeq \lim_{i \in I\op} \QCoh_d( X_i \times_X \mathfrak A ) \ . \]
		
		\item If $(\mathfrak A,\alpha)$ is the trivial $\bbG_m$-gerbe $\rB\bbG_m \times X$, then a quasi-coherent sheaf $\cF \in \QCoh(\rB\bbG_m \times X)$ is $d$-homogeneous if and only if it is of the form $\cG(d)$ for some $\cG \in \QCoh(X)$.
		In particular, the functor $\mathrm T_{X,d} \colon \QCoh(X) \to \QCoh(\rB\bbG_m \times X)$ restricts to an equivalence
		\[ \QCoh(X) \simeq \QCoh_d(\rB\bbG_m \times X) \ . \]
	\end{enumerate}
\end{lem}

\begin{proof}
	For point (1), it is enough to observe that
	\[ \delta_d(\cF) \simeq p_\ast(\mathrm{act}_\alpha^\ast(\cF)(-d)) \ . \]
	Now, $\mathrm{act}_\alpha^\ast$ is $t$-exact thanks to \cref{lem:act_t-exact}, while $p_\ast \colon \QCoh(\rB\bbG_m \times \mathfrak A)$ is $t$-exact because it is identified with evaluation at $0 \in \Z$ under the equivalence $\QCoh(\rB\bbG_m \times \mathfrak A) \simeq \Fun(\Z, \QCoh(\mathfrak A))$ (see \cref{lem:graded_spectra_standard_facts}).
	
	\medskip
	
	For point (2), observe first that since $\dSt$ is an $\infty$-topos the canonical map
	\[ \colim_{i \in I} X_i \times_X \mathfrak A \longrightarrow \mathfrak A \]
	is an equivalence.
	The conclusion now follows from the observation that if $f_i \colon \mathfrak A_i \coloneqq X_i \times_X \mathfrak A \to \mathfrak A$ is the canonical map, then the induced natural transformation
	\[ f_i^\ast \circ \delta_{d, \mathfrak A} \longrightarrow \delta_{d, \mathfrak A_i} \circ f_i^\ast \]
	is an equivalence.
	
	\medskip
	
	We are left to check point (3).
	Since $\delta_d$ is $\QCoh(X)$-linear, we can simply assume $X = \Spec(\bbS)$.
	In virtue of \cref{eg:act_trivial_gerbe}, we see that $\mathrm{act}_\alpha$ coincides in this case with the multiplication map
	\[ m \colon \rB\bbG_m \times \rB\bbG_m \longrightarrow \rB\bbG_m \ . \]
	Fix $\cF \in \QCoh(\rB\bbG_m)$ and for $d \in \Z$ write
	\[ \cG_d \coloneqq \mathrm{wt}_d(\cF) \in \Sp \ , \]
	so that the canonical map
	\[ \bigoplus_{k \in \Z} \cG_k(k) \longrightarrow \cF \]
	is an equivalence.
	It follows that
	\[ m^\ast(\cF) \simeq \bigoplus_{k \in \Z} \cG_k(k,k) \ . \]
	Therefore,
	\[ \delta_d(\cF) \simeq \bigoplus_{k \in \Z} p_\ast( \cG_d(k,k)(-d,0) ) \simeq \bigoplus_{k \in \Z} \cG_k \otimes p_\ast( \cO(k-d,k) ) \ , \]
	where $p \colon \rB\bbG_m \times \rB\bbG_m \to \rB\bbG_m$ is the second projection.
	Now, $p_\ast(\cO(k-d,k))$ is non-zero if and only if $k = d$.
	Besides, in this case we have $\cO(0,d) \simeq p^\ast(\cO(d))$, so the full faithfulness of $p^\ast$ (guaranteed by \cref{lem:graded_spectra_standard_facts}) immediately yields
	\[ p_\ast(\cO(0,d)) \simeq \cO(d) \ . \]
	Thus,
	\[ \delta_d(F) \simeq \cG_d(d) \simeq \mathrm{T}_d(\mathrm{wt}_d(\cF)) \ . \]
	The conclusion follows.
\end{proof}

\begin{cor}\label{lem:homogeneous_component_Azumaya}
	Let $X$ be a geometric derived stack and let $\alpha \colon \mathfrak A \to X$ be a $\bbG_m$-gerbe.
	Then for every integer $d \in \Z$, $\QCoh_d(\mathfrak A)$ is invertible in $\PrLomega_X$.
\end{cor}

\begin{proof}
	It follows from \cref{cor:etale_descent_Azumaya} that we can test invertibility in $\PrLomega_X$ smooth-locally on $X$.
	In other words, we can reduce ourselves to the case where $\mathfrak A$ is the trivial $\bbG_m$-gerbe.
	In this case, the statement follows directly from \cref{lem:d_homogeneous_basic_facts}-(3).
\end{proof}

As a consequence of \cref{lem:d_homogeneous_basic_facts}-(1) and of \cite[Lemma 5.2]{Bergh_Schnurer_Gerbes}, we obtain:

\begin{cor}
	Let $X$ be an underived scheme.
	Then the underlying triangulated category of $\QCoh_d(\mathfrak A)$ coincides with the triangulated category introduced in \cite{Bergh_Schnurer_Gerbes}.
\end{cor}

\begin{defin}
	Let $X$ be a derived stack and let $\alpha \in \rH^2_{\mathrm{\'et}}(X;\bbG_m)$.
	Let $(\mathfrak A, \alpha)$ be the associated $\bbG_m$-gerbe on $X$.
	Then the $\infty$-category of $\alpha$-twisted quasi-coherent sheaves on $X$ is
	\[ \QCoh_\alpha(X) \coloneqq \QCoh_{1}(\mathfrak A) . \]
\end{defin}

The following result has been obtained by M.\ Lieblich at the level of abelian categories \cite{Lieblich_Twisted_period_index}, and it has recently been extended to the derived categories by D.\ Bergh and O.\ Schn\"urer \cite{Bergh_Schnurer_Gerbes} for underived schemes.
We are now in position to generalize to the much vaster class of geometric derived stacks in the sense of Simpson \cite{Simpson_Algebraic_1996}.
We refer to \cite[Definition 2.8]{Porta_Yu_Higher_analytic_stacks_2014} for the precise inductive definition, or equivalently, to \cite[\S2.2.3]{HAG-II}.

\begin{rem}\label{rem:geometric_stacks}
	Without repeating the inductive definition of geometric derived stack here, let us recall to the reader that this class includes: derived schemes, derived algebraic spaces, derived (higher) Deligne-Mumford stacks, derived (higher) Artin stacks. The key difference between geometric derived stacks and Artin stacks lies in the very weak conditions on the diagonal that are imposed in the former; for instance, we allow derived stacks $X$ admitting a smooth cover $\{U_\alpha \to X\}_{\alpha \in I}$, where each $U_\alpha$ is a derived affine such that each fiber product $U_\alpha \times_X U_\beta$ are (not necessarily separated) derived Artin stacks (i.e.\ each $U_\alpha \times_X U_\beta$ has itself a smooth atlas and its diagonal is representable by -- not necessarily separated -- algebraic spaces).
\end{rem}

\begin{thm}[Lieblich, Bergh-Schn\"urer] \label{thm:qcoh_gerbe}
	Let $X$ be a geometric derived stack and let $\mathfrak A$ be a $\bbG_m$-gerbe on $X$.
	Then:
	\begin{enumerate}\itemsep=0.2cm
		\item Let $\alpha \in \rH^2_{\mathrm{\'et}}(X;\bbG_m)$ and let $(\mathfrak A, \alpha)$ be the associated $\bbG_m$-gerbe on $X$.
		Then pullback along the structural map $\pi \colon \mathfrak A \to X$ induces an equivalence
		\[ \pi^* \colon \QCoh(X) \longrightarrow \QCoh_{0}(\mathfrak A) . \]
		
		\item The $d$-homogeneous component functors $\delta_d$ introduced in \cref{construction:d_homogeneous_components} induce an equivalence
		\[ \mathrm D \colon \QCoh(\mathfrak A) \longrightarrow \prod_{d \in \Z} \QCoh_d(\mathfrak A) . \]
		
		\item The decomposition functor $\mathrm D$ of the previous point induces a fully faithful functor
		\[ \Perf(\mathfrak A) \longrightarrow \prod_{d \in \Z} \Perf_d(\mathfrak A) \ , \]
		whose essential image consists of those elements $(\cF_d)_{d \in \Z}$ such that for every map $f \colon T \to X$ from an affine derived scheme $T$, one has $f^\ast(F_d) = 0$ for all but a finite number of integers.
		
	\end{enumerate}
\end{thm}

\begin{proof}
	When $X$ is an underived scheme, statement (1) follows from \cite[Proposition 5.7]{Bergh_Schnurer_Gerbes}, and statement (2) follows from Theorem 5.4 in loc.\ cit.
	We sketch the proof in the general case.
	Thanks to \cref{lem:d_homogeneous_basic_facts}-(2), the three statements are smooth-local relative to $X$.
	We can therefore assume that $X$ is a derived affine scheme and that the gerbe is trivial.
	Since the functors $\pi^\ast$ and $\delta_d$ are $\QCoh(X)$-linear, we can further reduce to the case where $X = \Spec(\bbS)$.
	In this case, (1) follows from \cref{lem:graded_spectra_standard_facts}, and (2) is an immediate consequence of \cref{lem:d_homogeneous_basic_facts}-(3) and of \cref{thm:graded_spectra}.
	We are left checking statement (3).
	Notice that every perfect complex $\cF$ on $\rB\bbG_m$ is compact in $\QCoh(\rB\bbG_m)$.
	Using point (2), we see that the canonical map
	\[ \colim_{d \in \Z} \bigoplus_{-d \leqslant k \leqslant d} \mathrm{wt}_k(\cF)(k) \longrightarrow \cF \]
	is an equivalence.
	Since $\cF$ is compact, it follows that there exists $d \in \Z$ such that $\cF$ is a retract of
	\[ \bigoplus_{-d \leqslant k \leqslant d} \mathrm{wt}_k(\cF)(k) \ . \]
	Thus, $\cF$ has only finitely many nonzero weights.
\end{proof}

\begin{cor} \label{cor:GAGA_gerbes}
	Let $S$ be as at the beginning of \cref{sec:G-gerbes} and let $p \colon X \to S$ be a morphism in $\dSt$, where $X$ is a quasi-compact geometric derived stack.
	Let $\pi \colon \mathfrak A \to X$ be a $\bbG_m$-gerbe.
	Assume that $X$ is categorically proper, i.e.\ that the natural symmetric monoidal functor
	\[ \Perf(X) \longrightarrow \lim_n \Perf(X_n) \]
	is an equivalence.
	Then $\mathfrak A$ is categorically proper  as well, i.e.\ the canonical functor
	\[ \Perf(\mathfrak A) \longrightarrow \lim_n \Perf(\mathfrak A_n) \]
	is also a symmetric monoidal equivalence of stable symmetric monoidal $\infty$-categories.
\end{cor}

\begin{proof}
	Consider the natural map
	\[ \QCoh(\mathfrak A) \longrightarrow \lim_n \QCoh( \mathfrak A_n ) . \]
	Using the decomposition provided by \cref{thm:qcoh_gerbe}, we can rewrite it as
	\[ \prod_{d \in \Z} \QCoh_d(\mathfrak A) \longrightarrow \lim_n \prod_{d \in \Z} \QCoh_d(\mathfrak A_n) . \]
	Consider the diagrams
	\[ \begin{tikzcd}
		\Perf(\mathfrak A) \arrow{r} \arrow{d} & \lim_n \Perf(\mathfrak A_n) \arrow{d} \\
		\prod_{d \in \Z} \Perf_d(\mathfrak A) \arrow{r} & \lim_n \prod_{d \in \Z} \Perf_d(\mathfrak A_n) \ .
	\end{tikzcd} \]
	As a consequence of \cref{thm:qcoh_gerbe}-(2), we see that the vertical arrows are fully faithful.
	On the other hand, \cref{lem:homogeneous_component_Azumaya} implies that $\QCoh_\chi(\mathfrak A)$ is invertible in $\PrLomega_X$.
	In particular, it is smooth and proper and hence \cref{thm:GAGA_Morita} implies that the canonical map
	\[ \QCoh_d(\mathfrak A) \longrightarrow \lim_n \QCoh_d(\mathfrak A_n) \]
	is an equivalence, the limit being computed in $\PrLomega_X$.
	Passing to compact objects and taking the infinite product over all characters of $\bbG_m$, we deduce that the bottom horizontal map in the above square is an equivalence.
	Thus, the top horizontal map is fully faithful.
	We are therefore left to prove essential surjectivity.
	Fix an element $(\cF_n)_{n \geqslant 0} \in \lim_n \Perf(\mathfrak A_n)$.
	Since $X$ is quasi-compact, $\cF_0$ has only finitely many non-zero weights by \cref{thm:qcoh_gerbe}-(3).
	Notice that if $\delta_d(\cF_0) \simeq 0$, then $\delta_d(\cF_n) \simeq 0$ for every $n \geqslant 0$.
	Thus, it follows that
	\[ \cF \coloneqq \lim_{n \geqslant 0} \cF_n \in \prod_{d \in \Z} \Perf_d(\mathfrak A) \]
	has only finitely many non-zero weights.
	Therefore, \cref{thm:qcoh_gerbe}-(3) guarantees that it belongs to $\Perf(\mathfrak A)$.
\end{proof}

\subsection{Derived Azumaya algebra vs.\ $\bbG_m$-gerbes: a dictionary} \label{subsec:dictionary}

Derived Azumaya algebras and $\bbG_m$-gerbes are connected in a rather indirect way: both objects can be used to represent classes of $\rH^2\et(X;\bbG_m)$.
It is nevertheless possible to provide a direct dictionary between the two.
Since the appearance of the first version of this paper, G.\ Nocera and M.\ Pernice made this dictionary explicit in \cite{Nocera_Pernice_Twisted_sheaves}.
Their method converts a derived Azumaya algebra into a $\bbG_m$-gerbe by looking at the stack of its ``positive trivializations'', and Proposition 2.9 in \emph{loc.\ cit.} settles the following statement, that appeared as a conjecture in the first version of this paper (motivated by \cite[Corollary 5.9]{Bergh_Schnurer_Gerbes}):

\begin{conj}
	Let $X$ be a quasi-compact and quasi-separated scheme.
	Let $(\mathfrak A,\alpha)$ and $(\mathfrak A',\alpha')$ be two $\bbG_m$-gerbes on $X$.
	Let $d, d' \in \Z$ be two integers.
	Then there exists a canonical equivalence
	\[ \QCoh_d(\mathfrak A) \otimes_{\QCoh(X)} \QCoh_{d'}(\mathfrak A') \simeq \QCoh_{(d,d')}(\mathfrak A \times_X \mathfrak A') , \]
	where $\mathfrak A \times_X \mathfrak A'$ is considered as $\bbG_m \times \bbG_m$-gerbe.
	Furthermore let $\mathfrak A \star \mathfrak A'$ be $\bbG_m$-gerbe on $X$ classifying the product $\alpha \alpha' \in \rH^2_{\mathrm{\'et}}(X;\bbG_m)$.
	Then pulling back along the canonical projection map of \cite[Construction 3.8]{Bergh_Schnurer_Gerbes}
	\[ \rho_{\alpha,\alpha'} \colon \mathfrak A \times_X \mathfrak A' \longrightarrow \mathfrak A \star \mathfrak A' \]
	yields an equivalence
	\[ \rho_{\alpha,\alpha'}^* \colon \QCoh_{1}(\mathfrak A \star \mathfrak A') \simeq \QCoh_{(1, 1)}(\mathfrak A \times_X \mathfrak A') . \]
\end{conj}

\bibliographystyle{plain}
\bibliography{dahema}

\end{document}

%% file: newcommands_all.tex
\ifpersonal
\newcommand{\personal}[1]{\textcolor[rgb]{0,0,1}{(Personal: #1)}}
\newcommand{\todo}[1]{\textcolor{red}{(Todo: #1)}}
\else
\newcommand*{\personal}[1]{\ignorespaces}
\newcommand*{\todo}[1]{\ignorespaces}
\fi

\newcommand{\Z}{\mathbb Z}

\newcommand{\rB}{\mathrm B}
\newcommand{\rH}{\mathrm H}
\newcommand{\rK}{\mathrm K}
\newcommand{\rI}{\mathrm I}

\newcommand{\rR}{\mathrm R}

\newcommand{\fX}{\mathfrak X}

\newcommand{\cA}{\mathcal A}
\newcommand{\cB}{\mathcal B}
\newcommand{\cC}{\mathcal C}
\newcommand{\cD}{\mathcal D}
\newcommand{\cE}{\mathcal E}
\newcommand{\cF}{\mathcal F}
\newcommand{\cH}{\mathcal H}
\newcommand{\cG}{\mathcal G}

\newcommand{\cL}{\mathcal L}

\newcommand{\cO}{\mathcal O}

\newcommand{\cS}{\mathcal S}

\DeclareFontFamily{U}{BOONDOX-calo}{\skewchar\font=45 }
\DeclareFontShape{U}{BOONDOX-calo}{m}{n}{<-> s*[1.05] BOONDOX-r-calo}{}
\DeclareFontShape{U}{BOONDOX-calo}{b}{n}{<-> s*[1.05] BOONDOX-b-calo}{}
\DeclareMathAlphabet{\mathcalboondox}{U}{BOONDOX-calo}{m}{n}

\newcommand{\bbG}{\mathbb G}

\newcommand{\bbN}{\mathbb N}

\makeatletter
\let\save@mathaccent\mathaccent
\newcommand*\if@single[3]{%
	\setbox0\hbox{${\mathaccent"0362{#1}}^H$}%
	\setbox2\hbox{${\mathaccent"0362{\kern0pt#1}}^H$}%
	\ifdim\ht0=\ht2 #3\else #2\fi
}
\newcommand*\rel@kern[1]{\kern#1\dimexpr\macc@kerna}
\newcommand*\widebar[1]{\@ifnextchar^{{\wide@bar{#1}{0}}}{\wide@bar{#1}{1}}}
\newcommand*\wide@bar[2]{\if@single{#1}{\wide@bar@{#1}{#2}{1}}{\wide@bar@{#1}{#2}{2}}}
\newcommand*\wide@bar@[3]{%
	\begingroup
	\def\mathaccent##1##2{%
		\let\mathaccent\save@mathaccent
		\if#32 \let\macc@nucleus\first@char \fi
		\setbox\z@\hbox{$\macc@style{\macc@nucleus}_{}$}%
		\setbox\tw@\hbox{$\macc@style{\macc@nucleus}{}_{}$}%
		\dimen@\wd\tw@
		\advance\dimen@-\wd\z@
		\divide\dimen@ 3
		\@tempdima\wd\tw@
		\advance\@tempdima-\scriptspace
		\divide\@tempdima 10
		\advance\dimen@-\@tempdima
		\ifdim\dimen@>\z@ \dimen@0pt\fi
		\rel@kern{0.6}\kern-\dimen@
		\if#31
		\overline{\rel@kern{-0.6}\kern\dimen@\macc@nucleus\rel@kern{0.4}\kern\dimen@}%
		\advance\dimen@0.4\dimexpr\macc@kerna
		\let\final@kern#2%
		\ifdim\dimen@<\z@ \let\final@kern1\fi
		\if\final@kern1 \kern-\dimen@\fi
		\else
		\overline{\rel@kern{-0.6}\kern\dimen@#1}%
		\fi
	}%
	\macc@depth\@ne
	\let\math@bgroup\@empty \let\math@egroup\macc@set@skewchar
	\mathsurround\z@ \frozen@everymath{\mathgroup\macc@group\relax}%
	\macc@set@skewchar\relax
	\let\mathaccentV\macc@nested@a
	\if#31
	\macc@nested@a\relax111{#1}%
	\else
	\def\gobble@till@marker##1\endmarker{}%
	\futurelet\first@char\gobble@till@marker#1\endmarker
	\ifcat\noexpand\first@char A\else
	\def\first@char{}%
	\fi
	\macc@nested@a\relax111{\first@char}%
	\fi
	\endgroup
}
\makeatother

\newcommand{\Sh}{\mathrm{Sh}}

\newcommand{\Ab}{\mathrm{Ab}}
\newcommand{\DAb}{\cD(\Ab)}

\newcommand{\tauet}{\tau_\mathrm{\acute{e}t}}

\newcommand{\Mod}{\mathrm{Mod}}

\newcommand{\CAlg}{\mathrm{CAlg}}

\newcommand{\Cat}{\mathrm{Cat}}

\newcommand{\dAff}{\mathrm{dAff}}

\newcommand{\bfMap}{\mathbf{Map}}
\newcommand{\cHom}{\cH \mathrm{om}}
\newcommand{\cEnd}{\cE\mathrm{nd}}

\newcommand{\PrL}{\mathcal P \mathrm{r}^{\mathrm{L}}}

\newcommand{\Perf}{\mathrm{Perf}}
\newcommand{\QCoh}{\mathrm{QCoh}}
\newcommand{\dSt}{\mathrm{dSt}}

\newcommand{\bfPerf}{\mathbf{Perf}}

\newcommand{\Ind}{\mathrm{Ind}}

\newcommand{\PrLomega}{\mathcal P \mathrm r^{\mathrm L, \omega}}
\newcommand{\bbS}{\mathbb S}
\newcommand{\dAz}{\mathsf{dAz}}
\newcommand{\dBr}{\mathrm{dBr}}

\newcommand{\FunL}{\mathrm{Fun}^{\mathrm L}}

\newcommand{\NcSmPr}{\mathsf{SmPr}^{\mathsf{cat}}}

\newcommand{\PrLcg}{\mathcal P \mathrm r^{\mathrm L, \mathrm{cg}}}

\newcommand{\et}{_\mathrm{\acute{e}t}}

\newcommand{\inv}{^{-1}}
\newcommand{\id}{\mathrm{id}}

\newcommand{\op}{^\mathrm{op}}
\newcommand{\Cech}{\check{\mathcal C}}

\usetikzlibrary{decorations.markings} 
\tikzset{
  closed/.style = {decoration = {markings, mark = at position 0.5 with { \node[transform shape, xscale = .8, yscale=.4] {/}; } }, postaction = {decorate} },
  open/.style = {decoration = {markings, mark = at position 0.5 with { \node[transform shape, scale = .7] {$\circ$}; } }, postaction = {decorate} }
}

\DeclareMathOperator{\Fun}{Fun}

\DeclareMathOperator{\Map}{Map}

\DeclareMathOperator{\Pic}{Pic}

\DeclareMathOperator{\Sp}{Sp}

\DeclareMathOperator{\Spec}{Spec}
\DeclareMathOperator{\Spf}{Spf}

\DeclareMathOperator*{\colim}{colim}

%% file: GAGA_Grothendieck_conjecture_2.0.bbl
\def\cprime{$'$}
\begin{thebibliography}{10}

\bibitem{Alper_Hall_Rydh_Etale}
Jarod Alper, Jack Hall, and David Rydh.
\newblock The {\'e}tale local structure of algebraic stacks.
\newblock {\em arXiv preprint arXiv:1912.06162}, 2019.

\bibitem{Anel_Toen}
Mathieu Anel and Bertrand To\"{e}n.
\newblock D\'{e}nombrabilit\'{e} des classes d'\'{e}quivalences
  d\'{e}riv\'{e}es de vari\'{e}t\'{e}s alg\'{e}briques.
\newblock {\em J. Algebraic Geom.}, 18(2):257--277, 2009.

\bibitem{Beauville_Luroth}
Arnaud Beauville.
\newblock The {L}\"{u}roth problem.
\newblock In {\em Rationality problems in algebraic geometry}, volume 2172 of
  {\em Lecture Notes in Math.}, pages 1--27. Springer, Cham, 2016.

\bibitem{Beauville-Laszlo}
Arnaud Beauville and Yves Laszlo.
\newblock Un lemme de descente.
\newblock {\em C. R. Acad. Sci. Paris S\'er. I Math.}, 320(3):335--340, 1995.

\bibitem{Bergh_Schnurer_Gerbes}
Daniel Bergh and Olaf~M Schn{\"u}rer.
\newblock Decompositions of derived categories of gerbes and of families of
  brauer-severi varieties.
\newblock {\em arXiv preprint arXiv:1901.08945}, 2019.

\bibitem{Bhatt_algebraization_2014}
Bhargav Bhatt.
\newblock Algebraization and tannaka duality.
\newblock {\em arXiv preprint arXiv:1404.7483}, 2014.

\bibitem{Bondal_VdB}
A.~Bondal and M.~van~den Bergh.
\newblock Generators and representability of functors in commutative and
  noncommutative geometry.
\newblock {\em Mosc. Math. J.}, 3(1):1--36, 258, 2003.

\bibitem{Bouthier_Cesnavicius_Torsors_Loop}
Alexis Bouthier and K\c{e}stutis \v{C}esnavi\v{c}ius.
\newblock Torsors on loop groups and the hitchin fibration.
\newblock {\em arXiv preprint arXiv:1908.07480}, 2019.

\bibitem{Colliot_Thelene_Brauer_Grothendieck_group}
Jean-Louis Colliot-Th{\'e}l{\`e}ne and Alexei~N. Skorobogatov.
\newblock The {B}rauer-{G}rothendieck group.
\newblock Preprint, 2019.

\bibitem{de_Jong_Gabber}
A.~{J}ohan de~{J}ong.
\newblock A result of {G}abber.
\newblock Preprint, 2004.

\bibitem{Deninger_Proper_base_change}
Ch. Deninger.
\newblock A proper base change theorem for nontorsion sheaves in \'{e}tale
  cohomology.
\newblock {\em J. Pure Appl. Algebra}, 50(3):231--235, 1988.

\bibitem{Vistoli_Brauer_quotient_stack}
Dan Edidin, Brendan Hassett, Andrew Kresch, and Angelo Vistoli.
\newblock Brauer groups and quotient stacks.
\newblock {\em Amer. J. Math.}, 123(4):761--777, 2001.

\bibitem{Gaitsgory_1_affineness}
Dennis Gaitsgory.
\newblock Sheaves of categories and the notion of 1-affineness.
\newblock In {\em Stacks and categories in geometry, topology, and algebra},
  volume 643 of {\em Contemp. Math.}, pages 127--225. Amer. Math. Soc.,
  Providence, RI, 2015.

\bibitem{Geisser_Morin_Kernel_Brauer_Manin}
Thomas~H. Geisser and Baptiste Morin.
\newblock On the kernel of the brauer-manin pairing, 2021.

\bibitem{Grothendieck_Dix_expose}
J.~Giraud, A.~Grothendieck, S.~L. Kleiman, M.~Raynaud, and J.~Tate.
\newblock {\em Dix expos\'{e}s sur la cohomologie des sch\'{e}mas}, volume~3 of
  {\em Advanced Studies in Pure Mathematics}.
\newblock North-Holland Publishing Co., Amsterdam; Masson \& Cie, Editeur,
  Paris, 1968.

\bibitem{Giraud_Cohomologie_1971}
Jean Giraud.
\newblock {\em Cohomologie non ab\'elienne}.
\newblock Springer-Verlag, Berlin, 1971.
\newblock Die Grundlehren der mathematischen Wissenschaften, Band 179.

\bibitem{Rydh_Hall_Perfect_complexes}
Jack Hall and David Rydh.
\newblock Perfect complexes on algebraic stacks.
\newblock {\em Compos. Math.}, 153(11):2318--2367, 2017.

\bibitem{Hall_Rydh_Coherent}
Jack Hall and David Rydh.
\newblock Coherent {T}annaka duality and algebraicity of {H}om-stacks.
\newblock {\em Algebra Number Theory}, 13(7):1633--1675, 2019.

\bibitem{Halpern-Leistner_Preygel}
Daniel Halpern-Leistner and Anatoly Preygel.
\newblock Mapping stacks and categorical notions of properness.
\newblock {\em Compos. Math.}, 159(3):530----589, 2023.

\bibitem{Lurie_Brauer}
Mike Hopkins and Jacob Lurie.
\newblock On {B}rauer groups of {L}ubin-{T}ate spectra {I}.
\newblock Preprint, 2020.

\bibitem{Jannsen_Continuous}
Uwe Jannsen.
\newblock Continuous \'{e}tale cohomology.
\newblock {\em Math. Ann.}, 280(2):207--245, 1988.

\bibitem{kresh-mathur}
Andrew Kresh and Siddharth Mathur.
\newblock Formal {GAGA} for gerbes.
\newblock 2022.
\newblock Preprint, available on the author's webpage.

\bibitem{Lieblich_Twisted_period_index}
Max Lieblich.
\newblock Twisted sheaves and the period-index problem.
\newblock {\em Compos. Math.}, 144(1):1--31, 2008.

\bibitem{HTT}
Jacob Lurie.
\newblock {\em Higher topos theory}, volume 170 of {\em Annals of Mathematics
  Studies}.
\newblock Princeton University Press, Princeton, NJ, 2009.

\bibitem{Lurie_Higher_algebra}
Jacob Lurie.
\newblock Higher algebra.
\newblock Preprint, September 2017.

\bibitem{Lurie_SAG}
Jacob Lurie.
\newblock Spectral {A}lgebraic {G}eometry.
\newblock Preprint, 2018.

\bibitem{Milne_Etale_cohomology}
James~S. Milne.
\newblock {\em \'{E}tale cohomology}.
\newblock Princeton Mathematical Series, No. 33. Princeton University Press,
  Princeton, N.J., 1980.

\bibitem{Moulinos_Filtrations}
Tasos Moulinos.
\newblock The geometry of filtrations.
\newblock {\em arXiv preprint arXiv:1907.13562}, 2019.

\bibitem{Nocera_Pernice_Twisted_sheaves}
Guglielmo Nocera and Michele Pernice.
\newblock The derived brauer map via twisted sheaves.
\newblock {\em arXiv preprint arXiv:2205.07789}, 2022.

\bibitem{Porta_Yu_Higher_analytic_stacks_2014}
Mauro Porta and Tony~Yue Yu.
\newblock Higher analytic stacks and {GAGA} theorems.
\newblock {\em arXiv preprint arXiv:1412.5166}, 2014.

\bibitem{shioda2}
Tetsuji Shioda.
\newblock An example of unirational surfaces in characteristic {$p$}.
\newblock {\em Math. Ann.}, 211:233--236, 1974.

\bibitem{shioda1}
Tetsuji Shioda.
\newblock Algebraic cycles on certain {$K3$} surfaces in characteristic{$p$}.
\newblock In {\em Manifolds---{T}okyo 1973 ({P}roc. {I}nternat. {C}onf.,
  {T}okyo, 1973)}, pages 357--364. Univ. Tokyo Press, Tokyo, 1975.

\bibitem{Simpson_Algebraic_1996}
Carlos Simpson.
\newblock Algebraic (geometric) $ n $-stacks.
\newblock {\em arXiv preprint alg-geom/9609014}, 1996.

\bibitem{stacks-project}
The {Stacks Project Authors}.
\newblock {Stacks Project}.
\newblock \url{http://stacks.math.columbia.edu}, 2013.

\bibitem{Toen_Azumaya}
Bertrand To\"{e}n.
\newblock Derived {A}zumaya algebras and generators for twisted derived
  categories.
\newblock {\em Invent. Math.}, 189(3):581--652, 2012.

\bibitem{HAG-II}
Bertrand To{\"e}n and Gabriele Vezzosi.
\newblock Homotopical algebraic geometry. {II}. {G}eometric stacks and
  applications.
\newblock {\em Mem. Amer. Math. Soc.}, 193(902):x+224, 2008.

\bibitem{Totaro_Resolution_property}
Burt Totaro.
\newblock The resolution property for schemes and stacks.
\newblock {\em J. Reine Angew. Math.}, 577:1--22, 2004.

\end{thebibliography}
